\documentclass[12pt]{article}

\usepackage[font=footnotesize,margin=1cm]{caption}

\usepackage[margin=1in]{geometry}

\usepackage{setspace}
\usepackage[utf8]{inputenc} 
\usepackage[T1]{fontenc}    
\usepackage{hyperref}       
\usepackage{url}            
\usepackage{booktabs}       
\usepackage{amsfonts}       
\usepackage{nicefrac}       
\usepackage{microtype}      

\usepackage{amsmath,amssymb,bbm,stmaryrd,mathrsfs,url,amsthm}
\newtheorem{theorem}{Theorem}
\newtheorem{lemma}{Lemma}

\usepackage{amstext}
\usepackage{array,xcolor}
\usepackage{color,soul}
\usepackage{hyperref}
\usepackage{blkarray}

\usepackage[all,cmtip]{xy}
\usepackage{booktabs,makecell}

\usepackage[T1]{fontenc}
\usepackage[utf8]{inputenc}
\usepackage{multirow,bigdelim}
\usepackage{makeidx}
\usepackage{overpic}
\makeindex

\usepackage[export]{adjustbox}
\usepackage{graphicx,floatrow}
\usepackage{subcaption}
\usepackage{natbib}
\DeclareMathOperator*{\argmin}{arg\,min}

\newcommand{\change}{\color{red}}

\allowdisplaybreaks

\usepackage{algpseudocode,algorithm,algorithmicx}

\floatstyle{plain}
\newfloat{myalgo}{tbhp}{mya}
\newenvironment{Algorithm}[1][tbh]%
{\begin{myalgo}[#1]
\centering
\begin{minipage}{13.5cm}
\begin{algorithm}[H]}
{\end{algorithm}
\end{minipage}
\end{myalgo}}

\newtheorem{cor}{Corollary}

\newcommand{\R}{\mathbb{R}}

\newcommand{\vct}[1]{\boldsymbol{#1}}
\newcommand{\mtx}[1]{\boldsymbol{#1}}

\newcommand{\<}{\langle}

\DeclareMathOperator*{\minimize}{\text{minimize}}

\newcommand{\va}{\vct{a}}
\newcommand{\vb}{\vct{b}}
\newcommand{\vc}{\vct{c}}

\newcommand{\ve}{\vct{e}}

\newcommand{\vg}{\vct{g}}
\newcommand{\vh}{\vct{h}}

\newcommand{\vm}{\vct{m}}

\newcommand{\vp}{\vct{p}}
\newcommand{\vq}{\vct{q}}

\newcommand{\vt}{\vct{t}}

\newcommand{\vv}{\vct{v}}
\newcommand{\vw}{\vct{w}}
\newcommand{\vx}{\vct{x}}
\newcommand{\vy}{\vct{y}}
\newcommand{\vz}{\vct{z}}
\newcommand{\vzero}{\vct{0}}
\newcommand{\vone}{\vct{1}}
\newcommand{\ind}{\boldsymbol{1}}

\newcommand{\vxi}{\vct{\xi}}

\newcommand{\mA}{\mtx{A}}
\newcommand{\mB}{\mtx{B}}
\newcommand{\mC}{\mtx{C}}

\newcommand{\mI}{\mtx{I}}

\newcommand{\mM}{\mtx{M}}

\newcommand{\mQ}{\mtx{Q}}

\newcommand{\mW}{\mtx{W}}

\newcommand{\mSigma}{\mtx{\Sigma}}

\def\l{\ell}

\newcommand{\htrue}{\vh_0}

\newcommand{\mtrue}{\vm_0}

\newcommand{\wtrue}{\vw_0}
\newcommand{\xtrue}{\vx_0}
\newcommand{\ytrue}{\vy_0}

\newcommand{\hath}{\hat{\vh}}
\newcommand{\hatm}{\hat{\vm}}

\newcommand{\relu}{\text{relu}}
\newcommand{\diag}{\text{diag}}

\newcommand{\Wf}{\mW^{(1)}}
\newcommand{\Ws}{\mW^{(2)}}

\newcommand{\Wfh}{\Wf_{1,+,\vh}}
\newcommand{\Wsh}{\Wf_{2,+,\vh}}
\newcommand{\Wih}{\Wf_{i,+,\vh}}
\newcommand{\Wdh}{\Wf_{d,+,\vh}}
\newcommand{\Wdmfh}{\Wf_{d-1,+,\vh}}

\newcommand{\Wdmfho}{\Wf_{d-1,+,\htrue}}

\newcommand{\Wix}{\Wf_{i,+,\vx}}
\newcommand{\Wdx}{\Wf_{d,+,\vx}}
\newcommand{\Wdmfx}{\Wf_{d-1,+,\vx}}

\newcommand{\Wdmfxo}{\Wf_{d-1,+,\xtrue}}

\newcommand{\Wim}{\Ws_{i,+,\vm}}
\newcommand{\Wdm}{\Ws_{d,+,\vm}}
\newcommand{\Wdmfm}{\Ws_{d-1,+,\vm}}

\newcommand{\Wiy}{\Ws_{i,+,\vy}}
\newcommand{\Wdy}{\Ws_{d,+,\vy}}

\newcommand{\Gh}{\mathcal{G}^{(1)}}
\newcommand{\Gm}{\mathcal{G}^{(2)}}

\newcommand{\prodWx}{\prod_{i=d}^1\Wix}
\newcommand{\prodWxdmf}{\prod_{i=d-1}^1\Wix}

\newcommand{\prodWh}{\prod_{i=d}^1\Wih}

\newcommand{\prodWhdmf}{\prod_{i=d-1}^1\Wih}

\newcommand{\prodWy}{\prod_{i=s}^1\Wiy}
\newcommand{\prodWysmf}{\prod_{i=s-1}^1\Wiy}

\newcommand{\prodWm}{\prod_{i=s}^1\Wim}

\newcommand{\prodWmsmf}{\prod_{i=s-1}^1\Wim}

\newcommand{\Lambdaph}{{\bf \Lambda}^{(1)}}
\newcommand{\Lambdapm}{{\bf \Lambda}^{(2)}}

\newcommand{\Lambdah}{\Lambdaph_{d,+,\vh}}
\newcommand{\Lambdaho}{\Lambdaph_{d,+,\htrue}}
\newcommand{\Lambdahdmf}{\Lambdaph_{d-1,+,\vh}}
\newcommand{\Lambdahodmf}{\Lambdaph_{d-1,+,\htrue}}

\newcommand{\Lambdam}{\Lambdapm_{s,+,\vm}}
\newcommand{\Lambdamdmf}{\Lambdapm_{s-1,+,\vm}}
\newcommand{\Lambdamo}{\Lambdapm_{s,+,\mtrue}}
\newcommand{\Lambdamodmf}{\Lambdapm_{s-1,+,\mtrue}}

\newcommand{\Lambdax}{\Lambdaph_{d,+,\vx}}

\newcommand{\Lambdaxdmf}{\Lambdaph_{d-1,+,\vx}}

\newcommand{\Lambday}{\Lambdapm_{s,+,\vy}}
\newcommand{\Lambdaydmf}{\Lambdapm_{s-1,+,\vy}}

\newcommand{\htrans}{\vh^\intercal}

\newcommand{\mtrans}{\vm^\intercal}

\newcommand{\Bh}{\mB_{+,\vh}}
\newcommand{\Bx}{\mB_{+,\vx}}
\newcommand{\Cm}{\mC_{+,\vm}}
\newcommand{\Cy}{\mC_{+,\vy}}

\newcommand{\htilde}{\tilde{\vh}}

\newcommand{\mtilde}{\tilde{\vm}}

\newcommand{\xtilde}{\tilde{\vx}}
\newcommand{\ytilde}{\tilde{\vy}}

\newcommand{\mtildet}{\tilde{\vm}^\intercal}

\newcommand{\xtildet}{\tilde{\vx}^\intercal}

\newcommand{\htruetilde}{\tilde{\vh}_0}
\newcommand{\mtruetilde}{\tilde{\vm}_0}

\newcommand{\vfhmhomo}{\vv_{(\vh,\vm),(\htrue,\mtrue)}^{(1)}}
\newcommand{\vshmhomo}{\vv_{(\vh,\vm),(\htrue,\mtrue)}^{(2)}}

\newcommand{\vfhm}{\vv_{(\vh,\vm)}^{(1)}}
\newcommand{\vshm}{\vv_{(\vh,\vm)}^{(2)}}

\newcommand{\vfhmn}{\vv_{(\vh_n,\vm_n)}^{(1)}}
\newcommand{\vshmn}{\vv_{(\vh_n,\vm_n)}^{(2)}}

\newcommand{\thvm}{\vt_{(\vh,\vm)}}
\newcommand{\tfhm}{\vt_{(\vh,\vm)}^{(1)}}
\newcommand{\tshm}{\vt_{(\vh,\vm)}^{(2)}}

\newcommand{\tfhmn}{\vt_{(\vh_n,\vm_n)}^{(1)}}

\newcommand{\tfhmdw}{\vt_{(\vh,\vm)+\delta\vw}^{(1)}}
\newcommand{\tshmdw}{\vt_{(\vh,\vm)+\delta\vw}^{(2)}}

\newcommand{\thmhomo}{\vt_{(\vh,\vm),(\htrue,\mtrue)}}
\newcommand{\tfhmhomo}{\vt_{(\vh,\vm),(\htrue,\mtrue)}^{(1)}}
\newcommand{\tshmhomo}{\vt_{(\vh,\vm),(\htrue,\mtrue)}^{(2)}}

\newcommand{\tftilde}{\tilde{\vt}^{(1)}}
\newcommand{\tstilde}{\tilde{\vt}^{(2)}}

\newcommand{\tfhhotilde}{\tftilde_{\vh,\htrue}}

\newcommand{\tsmmotilde}{\tstilde_{\vm,\mtrue}}

\newcommand{\rf}{r^{(1)}}
\newcommand{\rs}{r^{(2)}}

\newcommand{\gfhm}{\vg_{1,(\vh,\vm)}}
\newcommand{\gshm}{\vg_{2,(\vh,\vm)}}

\newcommand{\bi}{\vb_i}
\newcommand{\ci}{\vc_i}
\newcommand{\bit}{\bi^\intercal}
\newcommand{\cit}{\ci^\intercal}

\newcommand{\thetaibar}{\bar{\theta}_{i}}
\newcommand{\thetaimbar}{\bar{\theta}_{i-1}}
\newcommand{\thetaobar}{\bar{\theta}_{0}}
\newcommand{\thetajbar}{\bar{\theta}_{j}}

\newcommand{\thetaibark}{\bar{\theta}_{i}^{(k)}}
\newcommand{\thetaimbark}{\bar{\theta}_{i-1}^{(k)}}
\newcommand{\thetaobark}{\bar{\theta}_{0}^{(k)}}
\newcommand{\thetajbark}{\bar{\theta}_{j}^{(k)}}

\newcommand{\thetaabark}{\bar{\theta}_{\ak}^{(k)}}

\newcommand{\thetaichek}{\check{\theta}_{i}^{(k)}}
\newcommand{\thetaimchek}{\check{\theta}_{i-1}^{(k)}}
\newcommand{\thetaochek}{\check{\theta}_{0}^{(k)}}
\newcommand{\thetajchek}{\check{\theta}_{j}^{(k)}}

\newcommand{\thetadmfbars}{\bar{\theta}_{d-1}^{(2)}}
\newcommand{\thetadbars}{\bar{\theta}_{d}^{(2)}}

\newcommand{\thetaibarf}{\bar{\theta}_{i}^{(1)}}

\newcommand{\thetaobarf}{\bar{\theta}_{0}^{(1)}}

\newcommand{\thetadmfbarf}{\bar{\theta}_{d-1}^{(1)}}
\newcommand{\thetadbarf}{\bar{\theta}_{d}^{(1)}}

\newcommand{\thetaibars}{\bar{\theta}_{i}^{(2)}}

\newcommand{\thetaobars}{\bar{\theta}_{0}^{(2)}}

\newcommand{\thetasbars}{\bar{\theta}_{s}^{(2)}}

\newcommand{\xif}{\xi^{(1)}}
\newcommand{\xis}{\xi^{(2)}}
\newcommand{\xik}{\xi^{(k)}}

\newcommand{\zetaf}{\zeta^{(1)}}
\newcommand{\zetas}{\zeta^{(2)}}
\newcommand{\zetak}{\zeta^{(k)}}

\newcommand{\rhof}{\rho_{d}^{(1)}}
\newcommand{\rhos}{\rho_{s}^{(2)}}
\newcommand{\rhok}{\rho_{\ak}^{(k)}}

\newcommand{\rhods}{\rho_{d}^{(2)}}
\newcommand{\rhodk}{\rho_{d}^{(k)}}

\newcommand{\deltaf}{\delta^{(1)}}
\newcommand{\deltas}{\delta^{(2)}}
\newcommand{\deltak}{\delta^{(k)}}

\newcommand{\nhi}{{n_{i}}}
\newcommand{\nhimf}{{n_{i-1}}}
\newcommand{\nhzero}{{n_{0}}}
\newcommand{\nhf}{{n_{1}}}
\newcommand{\nhd}{{n_{d}}}

\newcommand{\pgi}{{p_{i}}}
\newcommand{\pgimf}{{p_{i-1}}}
\newcommand{\pgzero}{{p_{0}}}
\newcommand{\pgf}{{p_{1}}}
\newcommand{\pgs}{{p_{s}}}

\newcommand{\af}{{a^{(1)}}}
\newcommand{\as}{{a^{(2)}}}
\newcommand{\ak}{{a^{(k)}}}

\newcommand{\wit}{\vw_{i}^\intercal}

\title{Global Guarantees for Blind Demodulation with Generative Priors}

\author{
 Paul Hand\thanks{p.hand@northeastern.edu, Department of Mathematics and College of Computer and Information Science, Northeastern University}\ \ and Babhru Joshi\thanks{babhru.joshi@rice.edu, Department of Computational and Applied Mathematics, Rice University}
}

\begin{document}

\maketitle

\begin{abstract}
	We study a deep learning inspired formulation for the blind demodulation problem, which is the task of recovering two unknown vectors from their entrywise multiplication. We consider the case where the unknown vectors are in the range of known deep generative models, $\mathcal{G}^{(1)}:\mathbb{R}^n\rightarrow\mathbb{R}^\ell$ and $\mathcal{G}^{(2)}:\mathbb{R}^p\rightarrow\mathbb{R}^\ell$. In the case when the networks corresponding to the generative models are expansive, the weight matrices are random and the dimension of the unknown vectors satisfy $\ell = \Omega(n^2+p^2)$, up to log factors, we show that the empirical risk objective has a favorable landscape for optimization. That is, the objective function has a descent direction at every point outside of a small neighborhood around four hyperbolic curves. We also characterize the local maximizers of the empirical risk objective and, hence, show that there does not exist any other stationary points outside of these neighborhood around four hyperbolic curves and the set of local maximizers. We also implement a gradient descent scheme inspired by the geometry of the landscape of the objective function. In order to converge to a global minimizer, this gradient descent scheme exploits the fact that exactly one of the hyperbolic curve corresponds to the global minimizer, and thus points near this hyperbolic curve have a lower objective value than points close to the other spurious hyperbolic curves. We show that this gradient descent scheme can effectively remove distortions synthetically introduced to the MNIST dataset.
	 
\end{abstract}

\section{Introduction}
We study the problem of recovering two unknown vectors $\xtrue \in \R^\l$ and $\wtrue \in \R^\l$ from observations $\ytrue \in \R^\l$ of the form
 \begin{equation}\label{eq:prob_statement}
 	\ytrue = \wtrue\odot\xtrue,
 \end{equation}
where $\odot$ is entrywise multiplication. This bilinear inverse problem (BIP) is known as the blind demodulation problem. BIPs, in general, have been extensively studied and include problems such as blind deconvolution/demodulation \citep{ahmed2012blind, stockham1975blind, kundur1996blind,aghasi2016sweep,aghasi2017branchHull}, phase retrieval \citep{fienup1982phase, candes2012solving, candes2013phaselift}, dictionary learning \citep{tosic2011dictionary}, matrix factorization \citep{hoyer2004non, lee2001algorithms}, and self-calibration \citep{ling2015self}. A significant challenge of BIP is the ambiguity of solutions. These ambiguities are challenging because they cause the set of solutions to be non-convex.

A common ambiguity, also shared by the BIP in \eqref{eq:prob_statement}, is the scaling ambiguity. That is any member of the set $\{c\wtrue,\tfrac{1}{c}\xtrue\}$ for $c \neq 0$ solves \eqref{eq:prob_statement}. In addition to the scaling ambiguity, this BIP is difficult to solve because the solutions are non-unique, even when excluding the scaling ambiguity. For example, $(\wtrue, \xtrue)$ and $(\vone , \wtrue \odot \xtrue)$ both satisfy \eqref{eq:prob_statement}. This structural ambiguity can be solved by assuming a prior model of the unknown vectors. In past works relating to blind deconvolution and blind demodulation \citep{ahmed2012blind, aghasi2017branchHull}, this structural ambiguity issue was addressed by assuming a subspace prior, i.e. the unknown signals belong to known subspaces. Additionally, in many applications, the signals are compressible or sparse with respect to a basis like a wavelet basis or the Discrete Cosine Transform basis, which can address this structural ambiguity issue. 

In contrast to subspace and sparsity priors, we address the structural ambiguity issue by assuming the signals $\wtrue$ and $\xtrue$ belong to the range of known generative models $\Gh:\R^n\rightarrow\R^\l$ and $\Gm: \R^p\rightarrow\R^\l$, respectively.  That is, we assume that $\wtrue = \Gh(\htrue)$ for some $\htrue \in \R^n$ and $\xtrue = \Gm(\mtrue)$ for some $\mtrue \in \R^p$. So, to recover the unknown vectors $\wtrue$ and $\xtrue$, we first recover the latent code variables $\htrue$ and $\mtrue$ and then apply $\Gh$ and $\Gm$ on $\htrue$ and $\mtrue$, respectively. Thus, the blind demodulation problem under generative prior we study is:
\begin{align*}
	\text{find } \vh\in\R^n \text{ and } \vm \in \R^p \text{, up to the scaling ambiguity, such that } \ytrue = \Gh(\vh)\odot\Gm(\vm).
\end{align*}

In recent years, advances in generative modeling of images \citep{karras2017progressiveGAN} has significantly increased the scope of using a generative model as a prior in inverse problems. Generative models are now used in speech synthesis \citep{dieleman2016wavenet}, image in-painting \citep{Iizuka2017image}, image-to-image translation \citep{Zhu2017translation}, superresolution \citep{Casper2017Superresolution}, compressed sensing \citep{bora2017DCS, Lohit2017DCS}, blind deconvolution \citep{Asim2018DBD}, blind ptychography \citep{Shamshad2018DPty}, and in many more fields. Most of these papers empirically show that using generative model as a prior to solve inverse problems outperform classical methods. For example, in compressed sensing, optimization over the latent code space to recover images from its compressive measurements have been empirically shown to succeed with 10x fewer measurements than classical sparsity based methods \citep{bora2017DCS}. Similarly, the authors of \cite{Asim2018DBD} empirically show that using generative priors in image debluring inverse problem provide a very effective regularization that produce sharp deblurred images from very blurry images.  

In the present paper, we use generative priors to solve the blind demodulation problem \eqref{eq:prob_statement}. The generative model we consider is the an expansive, fully connected, feed forward neural network with Rectified Linear Unit (ReLU) activation functions and no bias terms. Our main contribution is we show that the empirical risk objective function, for a sufficiently expansive random generative model, has a landscape favorable for gradient based methods to converge to a global minimizer. Our result implies that if the dimension of the unknown signals satisfy $\l = \Omega(n^2 + p^2)$, up to log factors, then the landscape is favorable. In comparison, classical sparsity based methods for similar BIPs like sparse blind demodulation \citep{Bresler2016sparse} and sparse phase retrieval \citep{li2013sparse} showed that exact recovery of the unknown signals is possible if the number of measurements scale quadratically, up to a log factor, w.r.t. the sparsity level of the signals. While we show a similar scaling of the number of measurements w.r.t. the latent code dimension, the latent code dimension can be smaller than the sparsity level for the same signal, and thus recovering the signal using generative prior would require less number of measurements.

\subsection{Main results}
We study the problem of recovering two unknown signals $\wtrue$ and $\xtrue$ in $\R^{\l}$ from observations $\ytrue =\wtrue\odot\xtrue$, where $\odot$ denotes entrywise product. We assume, as a prior, that the vectors $\wtrue$ and $\xtrue$ belong to the range of $d$-layer and $s$-layer neural networks $\Gh:\R^n\rightarrow\R^\l$ and $\Gm:\R^p\rightarrow\R^\l$, respectively. The task of recovering $\wtrue$ and $\xtrue$ is reduced to finding the latent codes $\htrue \in \R^{n}$ and $\mtrue \in \R^{p}$ such that $\Gh(\htrue) = \wtrue$ and $\Gm(\mtrue)=\xtrue$. More precisely, we consider the generative networks modeled by $\Gh(\vh) = \relu(\Wf_d\dots\relu(\Wf_2\relu(\Wf_1 \vh))\dots)$ and $\Gm(\vm) = \relu(\Ws_s\dots\relu(\Ws_2\relu(\Ws_1 \vm))\allowbreak\dots)$, where $\relu(\vx)= \max(\vx,\vzero)$ applies entrywise, $\Wf_{i} \in \R^{\nhi \times \nhimf }$ for $i = 1,\dots,d$ with $n = \nhzero<\nhf<\dots<\nhd=\l$, and $\Ws_{i} \in \R^{\pgi\times \pgimf}$ for $i = 1,\dots,s$ with $p = \pgzero<\pgf<\dots<\pgs=\l$.  The blind demodulation problem we consider is:
\begin{align*}
\text{Let: }& \ytrue \in \R^{\l}, \htrue \in \R^{n}, \mtrue\in\R^{p} \text{ such that }
\ytrue = \Gh(\htrue)\odot \Gm(\mtrue),\\
\text{Given: }& \Gh, \Gm \text{ and  measurements } \ytrue,\\
\text{Find: }& \htrue \text{ and } \mtrue \text{, up to the scaling ambiguity.} 
\end{align*}
In order to recover $\htrue$ and $\mtrue$, up to the scaling ambiguity, we consider the following empirical risk minimization program:
\begin{align}\label{eq:ELF}
	\minimize_{\vh\in\R^n,\vm\in\R^p}f(\vh,\vm):=\frac{1}{2}\left\|\Gh\left(\htrue\right) \odot \Gm\left(\mtrue\right)-\Gh\left(\vh\right) \odot \Gm\left(\vm\right)\right\|_2^2.
\end{align}
\vspace{-.1in}
\begin{figure}[H]
\floatsetup[subfigure]{captionskip=-10pt} 
\ffigbox[]{\hspace{-0pt}
    \begin{subfloatrow}[2]
      \ffigbox[]{%
      \begin{overpic}[width=0.4\textwidth,height=.35\textwidth,tics=1]{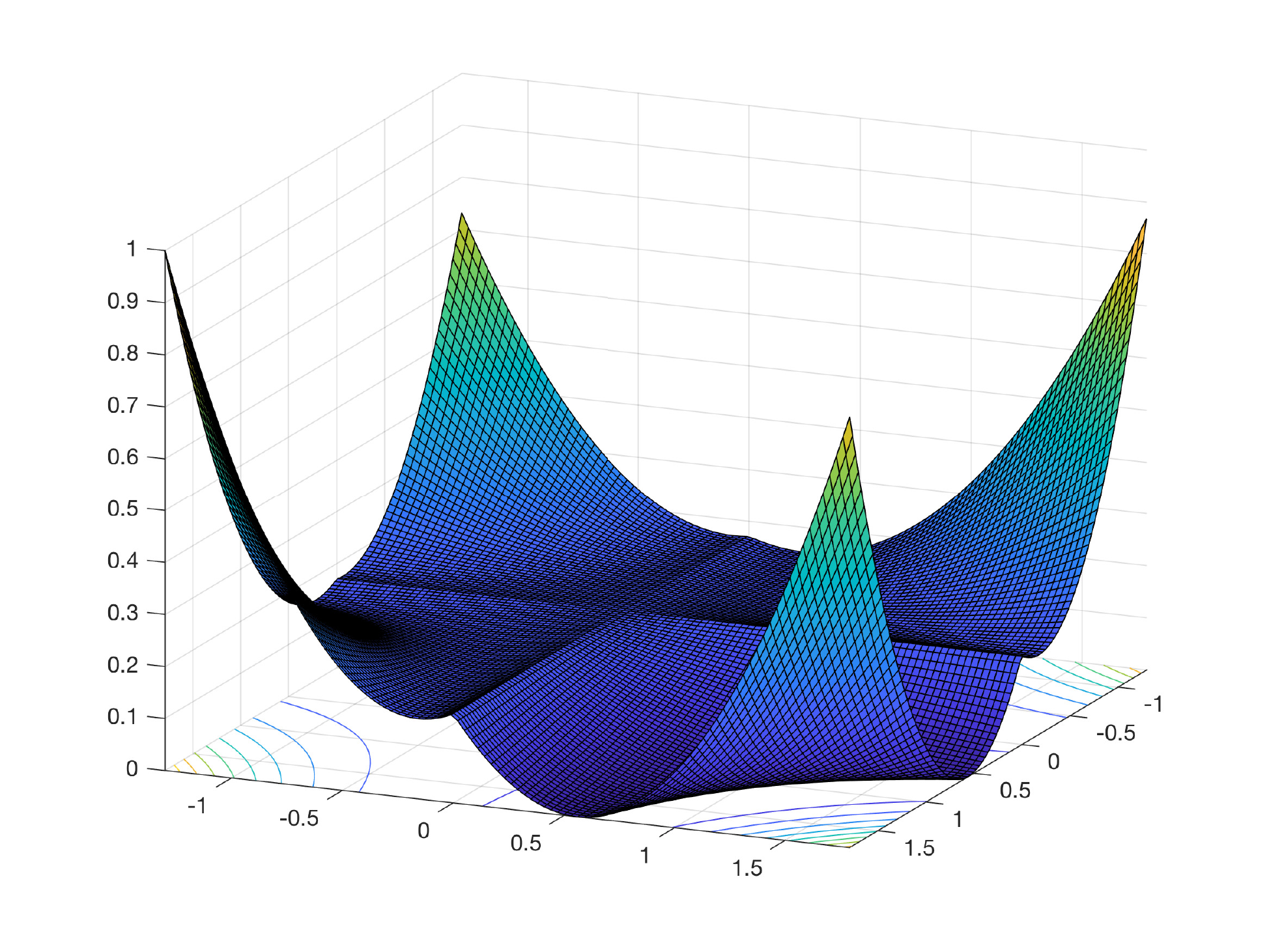}
   	  \end{overpic}
        }{\caption{Landscape of the empirical risk function.}\label{fig:surface_landscape}}
      \hspace{\fill}\ffigbox[]{\raisebox{.8cm}{
        \begin{overpic}[width=0.25\textwidth,height=0.25\textwidth,tics=1]{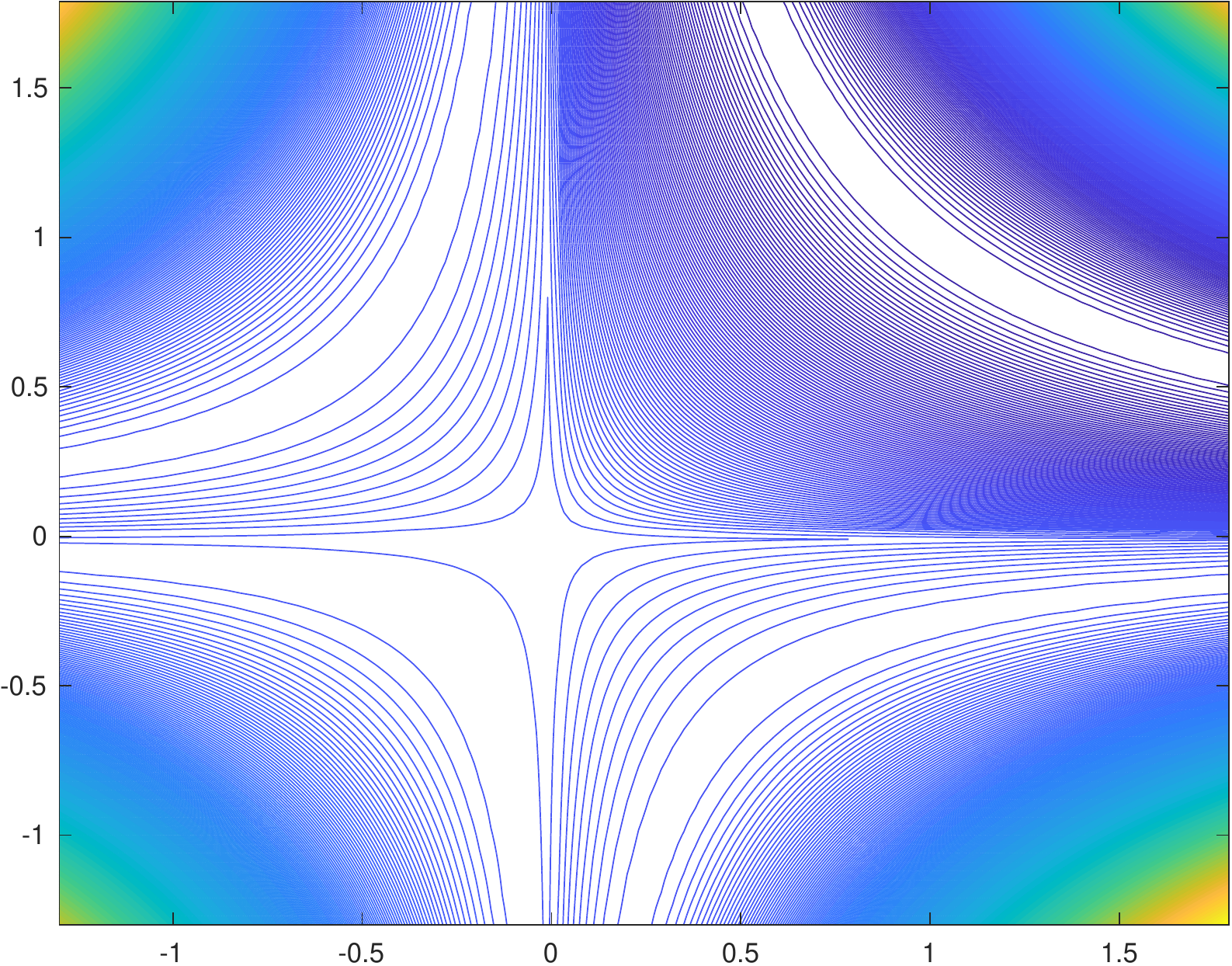}
	  \end{overpic}}}
	  {\caption{Note the four hyperbolic branches visible.}\label{fig:contour_landscape}}
    \end{subfloatrow}
}{\caption{Plots showing the landscape of the objective function with $\htrue = 1$ and $\mtrue = 1$.}

}
\end{figure}
\vspace{-10pt}

Figures \ref{fig:surface_landscape} and \ref{fig:contour_landscape} show the landscape of the objective function in the case when $\htrue =\mtrue = 1$, $s = d = 2$, the networks are expansive, and the weight matrices $\Wf_i$ and $\Ws_i$ contain i.i.d. Gaussian entries. Clearly, the objective function in \eqref{eq:ELF} is non-convex and, as a result, there does not exist a prior guarantee that gradient based methods will converge to a global minima. Additionally, the objective function does not contain any regularizer which are generally be used to resolve the scaling ambiguity, and thus every point in $\left\{(c\htrue, \tfrac{1}{c}\mtrue)|c>0\right\}$ is a global optima of \eqref{eq:ELF}. Nonetheless, we show that under certain conditions on the networks, the minimizers of \eqref{eq:ELF} are in the neighborhood of four hyperbolic curves, one of which is the hyperbolic curve containing the global minimizers.

In order to define these hyperbolic neighborhoods, let 
\begin{equation}\label{eq:recovery_area_hyper}
	\mathcal{A}_{\epsilon,(\htilde,\mtilde)} = \left\{(\vh,\vm)\in\R^{n\times p}\left|\exists\ c>0 \textit{ s.t. } \left\|(\vh,\vm)-\left(c\htilde,\frac{1}{c}\mtilde\right)\right\|_2\leq \epsilon\left\|\left(c\htilde,\frac{1}{c}\mtilde\right)\right\|_2 \right.\right\},
\end{equation}
where $(\htilde,\mtilde)\in\R^{n\times p}$ is fixed. This set is an $\epsilon$-neighborhood of the hyperbolic set $\{(c\htilde,\tfrac{1}{c}\mtilde)|c>0\}$. We show that the minimizers of \eqref{eq:ELF} are contained in the four hyperbolic sets given by $\mathcal{A}_{\epsilon,(\htrue,\mtrue)}$, $\mathcal{A}_{\epsilon,(-\rhof \htrue,\mtrue)}$, $\mathcal{A}_{\epsilon,(\htrue,-\rhos\mtrue)}$, and $\mathcal{A}_{\epsilon,(-\rhof\htrue,-\rhos\mtrue)}$. Here, $\epsilon$ depends on the expansivity and number of layers in the networks, and both $\rhof$ and $\rhos$ are positive constants close to 1. We also show that the points in the set $\{(\vh,\vzero)|\vh \in \R^n\} \cup \{(\vzero,\vm)|\vm \in \R^p\}$ are local maximizers. This result holds for networks with the following assumptions:
\begin{itemize}
\item[A1.] The weight matrices are random.
\item[A2.] The weight matrices of inner layers satisfy $\nhi = \Omega(\nhimf)$, up to a log factor, for $i = 1,\dots,d-1$ and $\pgi = \Omega(\pgimf)$, up to a log factor, for $i = 1,\dots,s-1$. 
\item[A3.] The weight matrices of the last layer for each generator satisfy $\l = \Omega(n_{d-1}^2+p_{s-1}^2)$, up to log factors.
\end{itemize}
Figures \ref{fig:surface_landscape} and \ref{fig:contour_landscape} show the landscape of the objective function and corroborate our findings.

\begin{theorem}[Informal]\label{thm:main_random_informal}
Let 
\begin{equation*}
\mathcal{A} = \mathcal{A}_{\epsilon,(\htrue,\mtrue)}\cup\mathcal{A}_{\epsilon,(-\rhof \htrue,\mtrue)}\cup\mathcal{A}_{\epsilon,(\htrue,-\rhos\mtrue)}\cup\mathcal{A}_{\epsilon,(-\rhof\htrue,-\rhos\mtrue)},
\end{equation*}
where $\epsilon>0$ depends on the expansivity of our networks and $\rhof,\rhos\rightarrow 1$ as $d,s\rightarrow\infty$, respectively. Suppose the networks are sufficiently expansive such that the number of neurons in the inner layers and the last layers satisfy assumptions A2 and A3, respectively. Then there exist a descent direction, given by one of the one-sided partial derivative of the objective function in \eqref{eq:ELF}, for every $(\vh,\vm)\notin \mathcal{A}\cup\{(\vh,\vzero)|\vh\in\R^n\}\cup\{(\vzero,\vm)|\vm\in\R^p\}$ with high probability. In addition, the elements of the set $\{(\vh,\vzero)|\vh\in\R^n\}\cup\{(\vzero,\vm)|\vm\in\R^p\}$ are local maximizers.
\end{theorem}
Our main result states that the objective function in \eqref{eq:ELF} does not have any spurious minimizers outside of the four hyperbolic neighborhoods. Thus, a gradient descent algorithm will converge to a point inside the four neighborhoods, one of which contains the global minimizers of \eqref{eq:ELF}. However, it may not guarantee convergence to a global minimizer and it may not resolve the inherent scaling ambiguity present in the problem. So, in order to converge to a global minimizer, we implement a gradient descent scheme that exploits the landscape of the objective function. That is, we exploit the fact that points near the hyperbolic curve corresponding to the global minimizer have a lower objective value than points that are close to the remaining three spurious hyperbolic curves. Second, in order to resolve the scaling ambiguity, we promote solutions that have equal $\l_2$ norm by normalizing the estimates in each iteration of the gradient descent scheme (See Section \ref{sec:algorithm}).

Theorem \ref{thm:main_random_informal} also provides a global guarantee of the landscape of the objective function in \eqref{eq:ELF} if the dimension of the unknown signals scale quadratically w.r.t. to the dimension of the latent codes, i.e. $\l = \Omega(n^2+p^2)$, up to log factors. Our result, which we get by enforcing generative priors may enjoy better sample complexity than classical priors like sparsity because: i) the same signals can have a latent code dimension that is smaller than its sparsity level w.r.t. to a particular basis, and ii) existing recovery guarantee of unstructured signals require number of measurements that scale quadratically with the sparsity level. Thus, our result may be less limiting in terms of sample complexity.

\subsection{Prior work on problems related to blind demodulation}

A common approach of solving the BIP in \eqref{eq:prob_statement} is to assume a subspace or sparsity prior on the unknown vectors. In these cases the unknown vectors $\wtrue$ and $\xtrue$ are assumed to be in the range of known matrices $\mB \in \R^{\l\times n}$ and $\mC \in \R^{\l\times p}$, respectively. In \cite{ahmed2012blind}, the authors assumed a subspace prior and cast the BIP as a linear low rank matrix recovery problem. They introduced a semidefinite program based on nuclear norm minimization to recover the unknown matrix. For the case where the rows of $\mB$ and $\mC$ are Fourier and Gaussian vectors, respectively, they provide a recovery guarantee that depend on the number of measurements as $\l = \Omega(n+p)$, up to log factors. However, because this method operates in the space of matrices, it is computationally prohibitively expensive. Another limitation of the lifted approach is that recovering a low rank and sparse matrix efficiently from linear observation of the matrix has been challenging. Recently, \cite{Bresler2016sparse} provided a recovery guarantee with near optimal sample complexity for the low rank and sparse matrix recovery problem using an alternating minimization method for a class of signals that satisfy a peakiness condition. However, for general signals the same work established a recovery result for the case where the number of measurements scale quadratically with the sparsity level.

In order to address the computational cost of working in the lifted case, a recent theme has been to introduce convex and non-convex programs that work in the natural parameter space. For example, in \cite{bahmani2016phase, goldstein2016phasemax}, the authors introduced PhaseMax, which is a convex program for phase retrieval that is based on finding  a simple convex relaxation via the convex hull of the feasibility set. The authors showed that PhaseMax enjoys rigorous recovery guarantee if a good anchor is available. This formulation was extended to the sparse case in \cite{hand2016compressedPhaseMax}, where the authors considered SparsePhaseMax and provided a recovery guarantee with optimal sample complexity. The idea of formulating a convex program using a simple convex relaxation via the convex hull of the feasibility set was used in the blind demodulation problem as well \citep{aghasi2017branchHull, aghasi2018L1BH}. In particular, \cite{aghasi2018L1BH} introduced a convex program in the natural parameter space for the sparse blind demodulation problem in the case where the sign of the unknown signals are known. Like in \cite{Bresler2016sparse}, the authors in \cite{aghasi2017branchHull} provide a recovery guarantee with optimal sample complexity for a class of signals. However, the result does not extend to signals with no constraints. Other approaches that operate in the natural parameter space are methods based on Wirtinger Flow. For example, in \cite{candes2014phase, wang2016solving, li2016rapid}, the authors use Wirtinger Flow and its variants to solve the phase retrieval and the blind deconvolution problem. These methods are non-convex and require a good initialization to converge to a global solution. However, they are simple to solve and enjoys rigorous recovery guarantees.

\subsection{Other related work}
In this paper, we consider the blind demodulation problem with the unknown signals assumed to be in the range of known generative models. Our work is motivated by experimental results in deep compressed sensing and deep blind deconvolution presented in \cite{bora2017DCS, Asim2018DBD} and theoretical work in deep compressed sensing presented in \cite{Hand2017DCS}. In \cite{bora2017DCS}, the authors consider the compressed sensing problem where, instead of a sparsity prior, a generative prior is considered. They used an empirical risk optimization program over the latent code space to recover images and empirical showed that their method succeeds with 10x fewer measurements than previous sparsity based methods. Following the empirical successes of deep compressed sensing, the authors in \cite{Hand2017DCS} provided a theoretical understanding for these successes by characterizing the landscape of the empirical risk objective function. In the random case with the layers of the generative model sufficiently expansive, they showed that every point outside of a small neighborhood around the true solution and a negative multiple of the true solution has a descent direction with high probability. Another instance where generative model currently outperforms sparsity based methods is in sparse phase retrieval \cite{hand2018DPR}. In sparse phase retrieval, current algorithms that enjoy a provable recovery guarantee of an unknown $n$-dimensional  $k$-sparse signal require at least $O(k^2\log n)$ measurements;  whereas, when assuming the unknown signal is an output of a known $d$-layer generator $\mathcal{G}:\R^k\rightarrow\R^n$, the authors in \cite{hand2018DPR} showed that, under favorable conditions on the generator and with at least $O(kd^2\log n)$ measurements, the empirical risk objective enjoys a favorable landscape. 

Similarly, in \cite{Asim2018DBD}, the authors consider the blind deconvolution problem where a generative prior over the unknown signal is considered. They empirically showed that using generative priors in the image deblurring inverse problem provide a very effective regularization that produce sharp deblurred images from very blurry images. The algorithm used to recovery these deblurred images is an alternating minimization approach which solves the empirical risk minimization with $\l_2$ regularization on the unknown signals. The $\l_2$ regularization promotes solution with least $\l_2$ norm and resolves the scaling ambiguity present in the blind deconvolution problem. We consider a related problem, namely the blind demodulation problem with a generative prior on the unknown signals, and show that under certain conditions on the generators the empirical risk objective has a favorable landscape.
\vspace{-.08cm}
\subsection{Notations}
Vectors and matrices are written with boldface, while scalars and entries of vectors are written in plain font.  We write $\boldsymbol{1}$ as the vector of all ones with dimensionality appropriate for the context. Let $\mathbb{S}^{n-1}$ be the unit sphere in $\R^n$. We write $\mI_n$ as the $n\times n$ identity matrix. For $\vx \in \R^K$ and $\vy \in \R^N$, $(\vx,\vy)$ is the corresponding vector in $\R^K \times \R^N$. Let $\relu(\vx) = \max(\vx,\vzero)$ apply entrywise for $\vx \in \R^n$. Let $\diag(\mW\vx>0)$ be the diagonal matrix that is $1$ in the $(i,i)$ entry if $(\mW\vx)_i>0$ and $0$ otherwise. Let $\mA\preceq\mB$ mean that $\mB-\mA$ is a positive semidefinite matrix. We will write $\gamma = O(\delta)$ to mean that there exists a positive constant $C$ such that $\gamma \leq C\delta$, where $\gamma$ is understood to be positive. Similarly we will write $c = \Omega(\delta)$ to mean that there exists a positive constant $C$ such that $c\geq C\delta$. When we say that a constant depends polynomially on $\epsilon^{-1}$, that means that it is at most $C\epsilon^{-k}$ for some positive $C$ and positive integer $k$. For notational convenience, we will write $\va  = \vb + O_1(\epsilon)$ if $\|\va -\vb\|\leq \epsilon$, where the norm is understood to be absolute value for scalars, the $\l_2$ norm for vectors, and the spectral norm for matrices.	

\section{Algorithm}\label{sec:algorithm}
In this section, we propose a gradient descent scheme that solves \eqref{eq:ELF}. The gradient descent scheme exploits the global geometry present in the landscape of the objective function in \eqref{eq:ELF} and avoids regions containing spurious minimizers. The gradient descent scheme is based on two observations. The first observation is that the minimizers of \eqref{eq:ELF} are close to four hyperbolic curves given by $\{(c\htrue,\tfrac{1}{c}\mtrue)|c>0\}$, $\{(-c\rhof\htrue,\tfrac{1}{c}\mtrue)|c>0\}$, $\{(c\htrue,-\tfrac{\rhods}{c}\mtrue)|c>0\}$,  and $\{(-c\rhof\htrue,-\tfrac{\rhos}{c}\mtrue)|c>0\}$, where $\rhof$ and $\rhos$ are close to 1. The second observation is that $f(c\htrue,\tfrac{1}{c}\mtrue)$ is less than $f(-c\htrue,\tfrac{1}{c}\mtrue)$, $f(c\htrue,-\tfrac{1}{c}\mtrue)$, and $f(-c\htrue,-\tfrac{1}{c}\mtrue)$ for any $c>0$. This is because the curve $\{(c\htrue,\tfrac{1}{c}\mtrue)|c>0\}$ corresponds to the global minimizer of \eqref{eq:ELF}.

We now introduce some quantities which are useful in stating the gradient descent algorithm.  For any $\vh \in \R^n$ and $\mW \in \R^{l\times n}$, define $\mW_{+,\vh} = \diag(\mW\vh>0)\mW$. That is, $\mW_{+,\vh}$ zeros out the rows of $\mW$ that do not have a positive inner product with $\vh$ and keeps the remaining rows. We will extend the definition of $\mW_{+,\vh}$ to each layer of weights $\Wf_{i}$ in our neural network. For $\Wf_{i}\in\R^{n_1\times n}$ and $\vh \in \R^{n}$, define $\Wfh := (\Wf_{1})_{+,\vh} = \diag(\Wf_{1}\vh>0)\Wf_1$. For each layer i > 1, define 
\[
\Wih = \diag(\Wf_{i}\Wf_{i-1,+,\vh}\dots\Wsh\Wfh\vh>0)\Wf_{i}.
\]
Lastly, define $\Lambdah := \prod_{i=d}^1\Wih$. Using the above notation, $\Gh(\vh)$ can be compactly written as $\Lambdah \vh$.  Similarly, we may write $\Gm(\vm)$ compactly as $\Lambdam\vm$.  

The gradient descent scheme is an alternating descent direction algorithm. We first pick an initial iterate $(\vh_1,\vm_1)$ such that $\vh_1\neq \vzero$ and $\vm_1 \neq \vzero$. At each iteration $i = 1, 2,\dots$, we first compare the objective value at $(\vh_1,\vm_1)$, $(-\vh_1,\vm_1)$, $(\vh_1,-\vm_1)$, and $(-\vh_1,-\vm_1)$ and reset $(\vh_1,\vm_1)$ to be the point with least objective value. Second we descend along a direction. We compute the descent direction $\tilde{\vg}_{1,(\vh,\vm)}$, given by the partial derivative of $f$ in \eqref{eq:ELF} w.r.t. $\vh$,
\begin{align*}
\tilde{\vg}_{1,(\vh,\vm)}:= {\Lambdah}^\intercal\left(\diag(\Lambdam\vm)^2\Lambdah\vh - \diag(\Lambdam\vm\odot\Lambdamo\mtrue)\Lambdaho\htrue\right)
\end{align*}
and take a step along this direction. Next, we compute the descent direction $\tilde{\vg}_{2,(\vh,\vm)}$, given by the partial derivative of $f$ in w.r.t. $\vm$,
\begin{align*}
\tilde{\vg}_{2,(\vh,\vm)}:={\Lambdam}^\intercal\left(\diag(\Lambdah\vh)^2\Lambdam\vm - \diag(\Lambdah\vh\odot\Lambdaho\htrue)\Lambdamo\mtrue\right).
\end{align*}
and again take a step along this direction. Lastly, we normalize the iterate so that at each iteration $i$ $\|\vh_i\|_2 = \|\vm_i\|_2$. We repeat this process until convergence. Algorithm \ref{alg:gradient_descent} outlines this process.
\vspace{-6mm}
\begin{Algorithm}[H]
\caption{Alternating descent algorithm for \eqref{eq:ELF}}
\label{alg:gradient_descent}
\begin{algorithmic}[1]
\Statex Input: Weight matrices, $\Wf_{i}$ and $\Ws_{i}$, observation $\ytrue$ and step size $\eta>0$.
\Statex Output: An estimate of a global minimizer of \eqref{eq:ELF} 
\State Choose an arbitrary point $(\vh_1,\vm_1)$ such that $\vh_1\neq \vzero$ and $\vm_1 \neq \vzero$
\For {$i = 1, 2,\dots$}:
\State $(\vh_i,\vm_i) \leftarrow \argmin(f(\vh_i,\vm_i),f(-\vh_i,\vm_i),f(\vh_i,-\vm_i),f(-\vh_i,-\vm_i))$
\State $h_{i+1} \leftarrow \vh_i - \eta \tilde{\vg}_{1,(\vh_i,\vm_i)}$,\hspace{.3cm} $m_{i+1} \leftarrow \vm_i - \eta \tilde{\vg}_{2,(\vh_{i+1},\vm_i)}$
\State $c \leftarrow \sqrt{\|\vh_{i+1}\|_2/\|\vm_{i+1}\|_2}$,\hspace{.3cm}$\vh_{i+1} \leftarrow \vh_{i+1}/c$,\hspace{.3cm}$\ \vm_{i+1} \leftarrow \vm_{i+1}\cdot c$ 
\EndFor
\end{algorithmic}
\end{Algorithm}
\vspace{-.6mm}

\section{Proof Outline}
We now present our main results which states that the objective function has a descent direction at every point outside of four hyperbolic regions. In order to state these directions, we first note that the partial derivatives of $f$ at a differentiable point $(\vh,\vm)$ are
\begin{align*}
	\nabla_{\vh} f(\vh,\vm) =\tilde{\vg}_{1,(\vh,\vm)} \text{ and } \nabla_{\vm} f(\vh,\vm) =\tilde{\vg}_{2,(\vh,\vm)}.
\end{align*}
The function $f$ is not differentiable everywhere because of the behavior of the RELU activation function in the neural network. However, since $\Gh$ and $\Gm$ are piecewise linear, $f$ is differentiable at $(\vh,\vm)+\delta\vw$ for all $(\vh,\vm)$ and $\vw$ and sufficiently small $\delta$. The directions we consider are $g_{1,(\vh,\vm)} \in \R^{n+ p}$ and $g_{2,(\vh,\vm)} \in \R^{n+p}$, where
\begin{align}\label{eq:directions}
\vg_{1,(\vh,\vm)}=\left[\begin{array}{c}\lim\limits_{\delta\rightarrow 0^+} \nabla_{\vh} f((\vh,\vm)+\delta \vw)\\ \vzero \end{array}\right], \  \vg_{2,(\vh,\vm)}=\left[\begin{array}{c} \vzero\\ \lim\limits_{\delta\rightarrow 0^+}\nabla_{\vm} f((\vh,\vm)+\delta \vw)  \end{array}\right], \text{ and}
\end{align}
$\vw$ is fixed. Let $D_{\vg}f(\vh,\vm)$ be the unnormalized one-sided directional derivative of $f(\vh,\vm)$ in the direction of $\vg$: $D_{\vg}f(\vh,\vm) = \lim_{t\rightarrow 0^+}\frac{f((\vh,\vm)+t\vg)-f(\vh,\vm)}{t}$. 
\begin{theorem}\label{thm:main_random}
Fix $\epsilon > 0$ such that $K_1(d^7s^2+d^2s^7)\epsilon^{1/4}<1$, $d\geq2$, and $s\geq2$. Assume  the networks satisfy assumptions A2 and A3. Assume for each $i = 1,\dots, d-1$, the entries of $\Wf_{i}$ are i.i.d. $\mathcal{N}(\vzero,\frac{1}{\nhi})$ and $i$th row of $\Wf_{d}$ satisfies $(\vw_{d}^{(1)})_i^\intercal = \vw^\intercal\cdot\vone_{\|\vw\|_2\leq3\sqrt{n_{d-1}/\l}}$ with $\vw \sim \mathcal{N}(0,\frac{1}{\l}\mI_{n_d-1})$. Similarly, assume for each $i = 1,\dots,s-1$, the entries of $\Ws_{i}$ are i.i.d.  $\mathcal{N}(0,\frac{1}{\pgi})$ and $i$th row of $\Ws_{s}$ satisfies $(\vw_{s}^{(2)})_i^\intercal = \vw^\intercal\cdot\vone_{\|\vw\|_2\leq3\sqrt{p_{s-1}/\l}}$ with $\vw \sim \mathcal{N}(\vzero,\frac{1}{\l}\mI_{p_s-1})$. Let $\mathcal{K}=\{(\vh,\vzero)\in\R^{n\times p}\left| \vh\in\R^{n}\right.\}\cup\{(\vzero,\vm)\in\R^{n\times p}\left| \vm\in\R^{p}\right.\}$ and $\mathcal{A} = \mathcal{A}_{K_2d^3s^3\epsilon^\frac{1}{4},(\htrue,\mtrue)}\cup \mathcal{A}_{K_2d^8s^3\epsilon^\frac{1}{4},\left(-\rhof\htrue,\mtrue\right)}\cup \mathcal{A}_{K_2d^3s^8\epsilon^\frac{1}{4},\left(\rhos\htrue,-\mtrue\right)}\cup \mathcal{A}_{K_2d^8s^8\epsilon^\frac{1}{4},\left(-\rhof\rhos\htrue,-\mtrue\right)}$. Then on an event of probability at least $1 - \sum_{i=1}^{d}\tilde{c}\nhi e^{-\gamma\nhimf}-\sum_{i=1}^{s}\tilde{c}\pgi e^{-\gamma\pgimf}-c e^{-\gamma \l}$ we have the following: for $(\htrue,\mtrue) \neq (\vzero,\vzero)$, and 
\begin{align*}
(\vh,\vm)\notin \mathcal{A} \cup \mathcal{K}
\end{align*}
the one-sided directional derivative of $f$ in the direction of $\vg = \vg_{1,(\vh,\vm)}$ or $\vg = \vg_{2,(\vh,\vm)}$, defined in \eqref{eq:directions}, satisfy $D_{-\vg}f(\vh,\vm)<0$. Additionally, for all $(\vh,\vm)\in \mathcal{K}$ and $(\vx,\vy)$
\begin{flalign*}
&D_{(\vx,\vy)}f(\vh,\vm)\leq0.
\end{flalign*}
Here, $\rhodk$ are positive numbers that converge to $1$ as $d\rightarrow\infty$, $c$ and $\gamma^{-1}$ are constants that depend polynomially on $\epsilon^{-1}$ and $\tilde{c}$, $K_1$, and $K_2$ are absolute constants.

\end{theorem}

We prove Theorem \ref{thm:main_random} by showing that neural networks with random weights satisfy two deterministic conditions. These conditions are the Weight Distributed Condition (WDC) and the joint Weight Distributed Condition (joint-WDC). The WDC is a slight generalization of the WDC introduced in \cite{Hand2017DCS}. We say a matrix $\mW\in\R^{\l\times n}$ satisfies the WDC with constants $\epsilon>0$ and $0<\alpha\leq 1$ if for all nonzero $\vx$, $\vy \in \R^k$,
\begin{align}
	\left\|\sum_{i=1}^\l \vone_{\vw_i\cdot\vx>0}\vone_{\vw_i\cdot\vy>0}\cdot\vw_i\vw_i^\intercal - \alpha \mQ_{\vx,\vy} \right\|\leq \epsilon, \text{ with } \mQ_{\vx,\vy} = \frac{\pi-\theta_0}{2\pi}\mI_{n}+\frac{\sin\theta_0}{2\pi}\mM_{\hat{\vx}\leftrightarrow\hat{\vy}},
\end{align}
where $\vw_i \in \R^n$ is the $i$th row of $\mW$; $\mM_{\hat{\vx}\leftrightarrow\hat{\vy}}\in\R^{n\times n}$ is the matrix such that $\hat{\vx}\rightarrow\hat{\vy}$, $\hat{\vy}\rightarrow\hat{\vx}$, and $\vz\rightarrow \vzero$ for all $\vz \in \text{span}(\{\vx,\vy\})^\perp$; $\hat{\vx} = \vx/\|\vx\|_2$ and $\hat{\vy} = \vy/\|\vy\|_2$; $\theta_0 = \angle(\vx,\vy)$; and $\vone_S$ is the indicator function on $S$. If $\vw_i \sim \mathcal{N}(\vzero,\frac{1}{\l}\mI_n)$ for all $i$, then an elementary calculation shows that $\mathbb{E}\left[\sum_{i=1}^\l \vone_{\vw_i\cdot\vx>0}\vone_{\vw_i\cdot\vy>0}\cdot\vw_i\vw_i^\intercal\right] = \mQ_{\vx,\vy}$ and if $\vx = \vy$ then $\mQ_{\vx,\vy}$ is an isometry up to a factor of $1/2$. Also, note that if $\mW$ satisfies WDC with constants $\epsilon$ and $\alpha$, then $\frac{1}{\sqrt{\alpha}}\mW$ satisfies WDC with constants $\epsilon/\alpha$ and $1$.

We now state the joint Weight Distributed Condition. We say that $\mB \in\R^{\l\times n}$ and $\mC \in \R^{\l\times p}$ satisfy joint-WDC with constants $\epsilon>0$ and $0<\alpha\leq 1$ if for all nonzero $\vh$, $\vx \in \R^{n}$ and nonzero $\vm$, $\vy \in \R^{p}$,
\begin{align}
&\left\|\mB_{+,\vh}^\intercal\diag\left(\mC_{+,\vm}\vm\odot\mC_{+,\vy}\vy\right)\mB_{+,\vx}- \frac{\alpha}{\l}\mtrans \mQ_{\vm,\vy}\vy\cdot \mQ_{\vh,\vx}\right\|\leq\frac{\epsilon}{\l}\|\vm\|_2\|\vy\|_2, \text{ and}\label{eq:pair-WDC1}\\
&\left\|\mC_{+,\vm}^\intercal\diag\left(\mB_{+,\vh}\vh\odot\mB_{+,\vx}\vx\right)\mC_{+,\vy}- \frac{\alpha}{\l}\htrans \mQ_{\vh,\vx}\vx\cdot \mQ_{\vm,\vy}\right\|\leq\frac{\epsilon}{\l}\|\vh\|_2\|\vx\|_2\label{eq:pair-WDC2}
\end{align}

 We analyze networks $\Gh$ and $\Gm$ where the weight matrices corresponding to the inner layers satisfy the WDC with constants $\epsilon>0$  and $1$ and for the two matrices corresponding to the outer layers, we assume that one of them satisfies WDC with constants $\epsilon$ and $0<\alpha_1\leq1$ and the other satifies WDC with constants $\epsilon$ and $0<\alpha_2\leq1$. We also assume that the two outer layer matrices satisfy joint-WDC with constants $\epsilon>0$ and $\alpha=\alpha_1\cdot\alpha_2$. We now state the main deterministic result:
\begin{theorem}\label{thm:main_determ}
Fix $\epsilon > 0$, $0<\alpha_1\leq 1$ and $0<\alpha_2\leq 1$ such that $K_1(d^7s^2+d^2s^7)\epsilon^{1/4}/(\alpha_1\alpha_2)<1$, $d\geq2$, and $s\geq2$. Let $\mathcal{K}=\{(\vh,\vzero)\in\R^{n\times p}\left| \vh\in\R^{n}\right.\}\cup\{(\vzero,\vm)\in\R^{n\times p}\left| \vm\in\R^{p}\right.\}$. Suppose that $\Wf_i \in\R^{\nhi\times \nhimf}$ for $i=1,\dots,d-1$ and $\Ws_i \in\R^{\pgi\times \pgimf}$ for $i=1,\dots,s-1$ satisfy the WDC with constant $\epsilon$ and $1$. Suppose $\Wf_d \in \R^{\l\times n_{d-1} }$ satisfy WDC with constants $\epsilon$ and $\alpha_1$, and $\Ws_s \in \R^{\l\times p_{s-1}}$ satisfy WDC with constants $\epsilon$ and $\alpha_2$. Also, suppose $\left(\Wf_d,\Ws_s\right)$ satisfy joint-WDC with constants $\epsilon$, $\alpha =\alpha_1\cdot\alpha_2$.  Let $\mathcal{K}=\{(\vh,\vzero)\in\R^{n\times p}\left| \vh\in\R^{n}\right.\}\cup\{(\vzero,\vm)\in\R^{n\times p}\left| \vm\in\R^{p}\right.\}$ and 
\begin{align*}
	\mathcal{A} =\ &\mathcal{A}_{K_2 d^3s^3\epsilon^\frac{1}{4}\alpha^{-1},(\htrue,\mtrue)}\cup \mathcal{A}_{K_2d^8s^3\epsilon^\frac{1}{4}\alpha^{-1},\left(-\rhof\htrue,\mtrue\right)}\cup \mathcal{A}_{K_2d^3s^8\epsilon^\frac{1}{4}\alpha^{-1},\left(\rhos\htrue,-\mtrue\right)}\\& \cup \mathcal{A}_{K_2d^8s^8\epsilon^\frac{1}{4}\alpha^{-1},\left(-\rhof\rhos\htrue,-\mtrue\right)}.
\end{align*}
Then, for $(\htrue,\mtrue) \neq (\vzero,\vzero)$, and 
\begin{align*}
&(\vh,\vm)\notin\mathcal{A}\cup \mathcal{K}
\end{align*}
the one-sided directional derivative of $f$ in the direction of $\vg = \vg_{1,(\vh,\vm)}$ or $\vg = \vg_{2,(\vh,\vm)}$ satisfy $D_{-\vg}f(\vh,\vm)<0$. Additionally, for all $(\vh,\vm)\in \mathcal{K}$ and for all $(\vx,\vy)$
\begin{flalign*}
&D_{(\vx,\vy)}f(\vh,\vm)\leq0.
\end{flalign*}
Here, $\rhodk$ are positive numbers that converge to $1$ as $d\rightarrow\infty$, and $K_1$, and $K_2$ are absolute constants.
\end{theorem}

We prove the theorems by showing that the descent directions $\vg_{1,(\vh,\vm)}$ and $\vg_{2,(\vh,\vm)}$ concentrate around its expectation and then characterize the set of points where the corresponding expectations are simultaneously zero. The outline of the proof is:
\begin{itemize}
\item The WDC and joint-WDC imply that the one-sided partial directional derivatives of $f$ concentrate uniformly for all non-zero $\vh$, $\htrue \in \R^n$ and $\vm$, $\mtrue \in \R^p$ around continuous vectors $\tfhmhomo$ and  $\tshmhomo$, respectively, defined in equations \eqref{eq:tfhmhomo} and \eqref{eq:tshmhomo} in the Appendix.
\item Direct analysis show that $\tfhmhomo$ and $\tshmhomo$ are simultaneously approximately zero around the four hyperbolic sets $\mathcal{A}_{\epsilon,(\htrue,\mtrue)}$, $\mathcal{A}_{\epsilon,(-\rhof \htrue,\mtrue)}$, $\mathcal{A}_{\epsilon,(\htrue,-\rhos\mtrue)}$, and $\mathcal{A}_{\epsilon,(-\rhof\htrue,-\rhos\mtrue)}$, where $\epsilon$ depends on the expansivity and number of layers in the networks, and both $\rhof$ and $\rhos$ are positive constants close to 1 and depends on the number of layers in the two neural networks as well.
\item Using sphere covering arguments, Gaussian and truncated Gaussian matrices with appropriate dimensions satisfy the WDC and joint-WDC conditions. 
\end{itemize}
The full proof of Theorem \ref{thm:main_determ} is provided in the Appendix.

\section{Numerical Experiment}
We now empirically show that Algorithm \ref{alg:gradient_descent} can remove distortions present in the dataset. We consider the image recovery task of removing distortions that were synthetically introduced to the MNIST dataset. The distortion dataset contain 8100  images of size $28 \times 28$ where the distortions are generated using a 2D Gaussian function, $g(x,y) = e^{-\tfrac{(x-c)^2+(y-c)^2}{\sigma}}$, where $c$ is the center and $\sigma$ controls its tail behavior. For each of the 8100 image, we fix $c$ and $\sigma$, which vary uniformly in the intervals $[-3,3]$ and $[20,35]$, respectively, and $x$ and $y$ are in the interval $[-5,5]$. Prior to training the generators, the images in the MNIST dataset and the distortion dataset were resized to $64\times 64$ images. We used  DCGAN \citep{Radford2016DCGAN} with a learning rate of 0.0002 and latent code dimension of 50 to train a generator, $\Gm$, for the distortion images. Similarly, we used the DCGAN with learning rate of 0.0002 and latent code dimension of 100 to train a generator, $\Gh$, for the MNIST images.  Finally, a distorted image $\ytrue$ is generated via the  pixelwise multiplication of an image $\wtrue$ from the MNIST dataset and an image $\xtrue$ from the distortion dataset, i.e. $\ytrue = \wtrue\odot\xtrue$.
\begin{figure}[H]
\includegraphics[scale=.22]{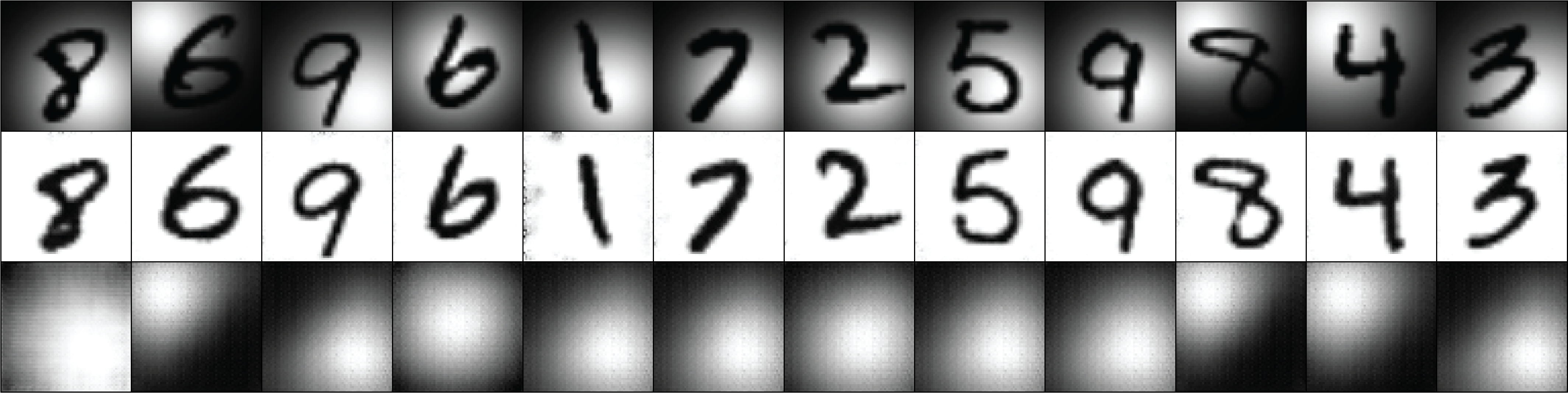}
\caption{The figure shows the result removing distortion in an image by solving \eqref{eq:ELF} using Algorithm \ref{alg:gradient_descent}. The top row corresponds to the input distorted image. The second and third row corresponds to the images recovered using empirical risk minimization.}
\label{fig:mnist_dist}
\end{figure}
\vspace{-10pt}

Figure \ref{fig:mnist_dist} shows the result of using Algorithm \ref{alg:gradient_descent} to remove distortion from $\ytrue$. In the implementation of Algorithm \ref{alg:gradient_descent}, $\tilde{\vg}_{1,(\vh_i,\vm_i)}$ and $\tilde{\vg}_{1,(\vh_i,\vm_i)}$ corresponds to the partial derivatives of $f$ with the generators as $\Gh$ and $\Gm$. We used the Stochastic Gradient Descent algorithm with the step size set to 1 and momentum set to 0.9. For each image in the first row of Figure \ref{fig:mnist_dist}, the corresponding images in the second and third rows are the output of Algorithm \ref{alg:gradient_descent} after 500 iterations.

\bibliographystyle{IEEEtran}
\bibliography{Bibliography}
\newpage
\section{Appendix}

Let $\angle(\vh,\htrue) = \thetaobarf$ and $\angle(\vm,\mtrue) = \thetaobars$ for non-zero $\vh,\htrue \in \R^n$ and $\vm,\mtrue \in \R^p$. In order to understand how the operators $\vh\rightarrow \Wf_{+,\vh}\vh$ and $\vm\rightarrow \Ws_{+,\vm}\vm$ distort angles, we define 
\begin{equation}\label{eq:g}
	g(\theta) = \cos^{-1}\left(\frac{(\pi-\theta)\cos\theta+\sin\theta}{\pi}\right).
\end{equation}
Also, for a fixed $\vp,\ \vq \in \R^n$, define
\begin{equation}\label{eq:tpq}
\tilde{t}_{\vp,\vq}^{(k)} := \frac{1}{2^\ak}\left[\left(\prod_{i=0}^{\ak-1}\frac{\pi-\thetaibark}{\pi}\right)\vq+\sum_{i=0}^{\ak-1}\frac{\sin\thetaibark}{\pi}\left(\prod_{j=i+1}^{\ak-1}\frac{\pi-\thetajbark}{\pi}\right)\frac{\|\vq\|_2}{\|\vp\|_2}\vp \right],
\end{equation}
where $\thetaibark = g(\thetaimbark)$ for $g$ given by \eqref{eq:g}, $\thetaobark = \angle(\vp,\vq)$, $\af=d$, and $\as = s$. 

\subsection{Proof of Deterministic Theorem}

\begin{theorem}[Also Theorem \ref{thm:main_determ}]\label{thm:main_determ_supplementary}
Fix $\epsilon > 0$, $0<\alpha_1\leq 1$ and $0<\alpha_2\leq 1$ such that $K_1(d^7s^2+d^2s^7)\epsilon^{1/4}/(\alpha_1\alpha_2)<1$, $d\geq2$, and $s\geq2$. Let $\mathcal{K}=\{(\vh,\vzero)\in\R^{n\times p}\left| \vh\in\R^{n}\right.\}\cup\{(\vzero,\vm)\in\R^{n\times p}\left| \vm\in\R^{p}\right.\}$. Suppose that $\Wf_i \in\R^{\nhi\times \nhimf}$ for $i=1,\dots,d-1$ and $\Ws_i \in\R^{\pgi\times \pgimf}$ for $i=1,\dots,s-1$ satisfy the WDC with constant $\epsilon$ and $1$. Suppose $\Wf_d \in \R^{\l\times n_{d-1} }$ satisfy WDC with constants $\epsilon$ and $\alpha_1$, and $\Ws_s \in \R^{\l\times p_{s-1}}$ satisfy WDC with constants $\epsilon$ and $\alpha_2$. Also, suppose $\left(\Wf_d,\Ws_s\right)$ satisfy joint-WDC with constants $\epsilon$, $\alpha =\alpha_1\cdot\alpha_2$.  Let $\mathcal{K}=\{(\vh,\vzero)\in\R^{n\times p}\left| \vh\in\R^{n}\right.\}\cup\{(\vzero,\vm)\in\R^{n\times p}\left| \vm\in\R^{p}\right.\}$ and $\mathcal{A} = \mathcal{A}_{K_2 d^3s^3\epsilon^\frac{1}{4}\alpha^{-1},(\htrue,\mtrue)}\cup \mathcal{A}_{K_2d^8s^3\epsilon^\frac{1}{4}\alpha^{-1},\left(-\rhof\htrue,\mtrue\right)}\cup \mathcal{A}_{K_2d^3s^8\epsilon^\frac{1}{4}\alpha^{-1},\left(\rhos\htrue,-\mtrue\right)}\cup \mathcal{A}_{K_2d^8s^8\epsilon^\frac{1}{4}\alpha^{-1},\left(-\rhof\rhos\htrue,-\mtrue\right)}$. Then, for $(\htrue,\mtrue) \neq (\vzero,\vzero)$, and 
\begin{align*}
&(\vh,\vm)\notin\mathcal{A}\cup \mathcal{K}
\end{align*}
the one-sided directional derivative of $f$ in the direction of $\vg = \vg_{1,(\vh,\vm)}$ or $\vg = \vg_{2,(\vh,\vm)}$ satisfy $D_{-\vg}f(\vh,\vm)<0$. Additionally, for all $(\vh,\vm)\in \mathcal{K}$ and for all $(\vx,\vy)$
\begin{flalign*}
&D_{(\vx,\vy)}f(\vh,\vm)\leq0.
\end{flalign*}
Here, $\rhodk$ are positive numbers that converge to $1$ as $d\rightarrow\infty$, and $K_1$, and $K_2$ are absolute constants.
\end{theorem}
\begin{proof}
Recall that 
\begin{align*}
		&\vfhmhomo = \left\{
		\begin{array}{ll}
			\nabla_{\vh} f(\vh,\vm) & G \text{ is differentiable at } (\vh,\vm),\\
			\lim_{\delta\rightarrow 0^+} \nabla_{\vh} f((\vh,\vm)+\delta \vw) & \text{ otherwise }
		\end{array}
		\right.\\
		&\vshmhomo = \left\{
		\begin{array}{ll}
			\nabla_{\vm} f(\vh,\vm) & G \text{ is differentiable at } (\vh,\vm),\\
			\lim_{\delta\rightarrow 0^+} \nabla_{\vm} f((\vh,\vm)+\delta \vw) & \text{ otherwise }
		\end{array}
		\right.
\end{align*}
where $G(\vh,\vm)$ is differentiable at $(\vh,\vm)+\delta\vw$ for sufficiently small $\delta$.  Such a $\delta$ exists because of piecewise linearity of $G(\vh,\vm)$ and any such $\vw$ can be arbitrarily selected. Also, recall that for any differentiable point $(\vh,\vm)$, we have
\begin{align*}
	\nabla_{\vh} f(\vh,\vm) =& \left(\Lambdah\right)^{\top}\diag\left(\Lambdam\vm\right)^2\Lambdah\vh\notag\\
	& -\left(\Lambdah\right)^{\top}\diag\left(\Lambdam\vm\odot\Lambdamo\mtrue\right)\Lambdaho\htrue,\\
	\nabla_{\vm} f(\vh,\vm) =& \left(\Lambdam\right)^{\top}\diag\left(\Lambdah\vh\right)^2\Lambdam\vm\notag\\
	&-\left(\Lambdam\right)^{\top}\diag\left(\Lambdah\vh\odot\Lambdaho\htrue\right)\Lambdamo\mtrue.
\end{align*}
Let 
\begin{align}
&\gfhm = \left[\begin{array}{c}\vfhmhomo\\ \vzero \end{array}\right]\in\R^{n\times p},\nonumber\\
&\gshm = \left[\begin{array}{c}\vzero\\ \vshmhomo \end{array}\right]\in\R^{n\times p},\nonumber\\
&\tfhmhomo =\frac{\alpha}{2^{d+s}\l}\|\vm\|_2^2\vh-\frac{\alpha}{\l}\mtrans\tsmmotilde\tfhhotilde,\label{eq:tfhmhomo}\\
&\tshmhomo =\frac{\alpha}{2^{d+s}\l}\|\vh\|_2^2\vm-\frac{\alpha}{\l}\htrans\tfhhotilde\tsmmotilde,\label{eq:tshmhomo}\\
&S_{\epsilon,(\htrue,\mtrue)}^{(1)} = \left\{(\vh,\vm)\in\R^{n\times p}\backslash\mathcal{K} \left|\frac{\|\tfhmhomo\|_2}{\|\vm\|_2 } \leq\frac{\epsilon\max(\|\vh\|_2\|\vm\|_2,\|\htrue\|_2\|\mtrue\|_2)}{2^{d+s}\l}\right.\right\},\nonumber\\
&S_{\epsilon,(\htrue,\mtrue)}^{(2)} = \left\{(\vh,\vm)\in\R^{n\times p}\backslash\mathcal{K} \left|\frac{\|\tshmhomo\|_2}{\|\vh\|_2 } \leq\frac{\epsilon\max(\|\vh\|_2\|\vm\|_2,\|\htrue\|_2\|\mtrue\|_2)}{2^{d+s}\l}\right.\right\}\nonumber,
\end{align}
where $\tsmmotilde$ and $\tfhhotilde$ is as defined in \eqref{eq:tpq}. For brevity of notation, write $\vfhm = \vfhmhomo$, $\vshm =\vshmhomo$, $\thvm = \thmhomo$, $\tfhm = \tfhmhomo$ and $\tshm = \tshmhomo$.

Since $\Wf_i \in\R^{\nhi\times \nhimf}$ for $i=1,\dots,d-1$ and $\Ws_i \in\R^{\pgi\times \pgimf}$ for $i=1,\dots,s-1$ satisfy the WDC with constant $\epsilon$ and $1$, $\Wf_d \in \R^{\l\times n_{d-1} }$ satisfy WDC with constants $\epsilon$ and $\alpha_1$, $\Ws_s \in \R^{\l\times p_{s-1}}$ satisfy WDC with constants $\epsilon$ and $\alpha_2$, and $\left(\Wf_d,\Ws_s\right)$ satisfy joint-WDC with constants $\epsilon$, $\alpha =\alpha_1\cdot\alpha_2$, we have lemma \ref{lem:main_concentration} implying for all nonzero $\vh,\ \htrue \in \R^{n}$ and nonzero $\vm,\ \mtrue \in \R^{p}$
\begin{align}
	\|\nabla_{\vh}f(\vh,\vm) - \tfhm \|_2 \leq K\frac{d^3s^3\sqrt{\epsilon}}{2^{d+s}\l}\max(\|\vh\|_2\|\vm\|_2,\|\htrue\|_2\|\mtrue\|_2)\|\vm\|_2,\label{eq:gradh_con} \\
	\|\nabla_{\vm}f(\vh,\vm) - \tshm\|_2 \leq K\frac{d^3s^2\sqrt{\epsilon}}{2^{d+s}\l}\max(\|\vh\|_2\|\vm\|_2,\|\htrue\|_2\|\mtrue\|_2)\|\vh\|_2\label{eq:gradx_con}.
\end{align}
Thus, we have, for all nonzero $\vh,\ \htrue \in \R^{n}$ and nonzero $\vm,\ \mtrue \in \R^{p}$,
\begin{align*}
	\|\vfhm-\tfhm\|_2 &= \lim_{\delta\rightarrow0^{+}} \|\nabla_{\vh}f\left((\vh,\vx)+\delta\vw\right) -\tfhmdw\|_2\\
	&\leq K\frac{d^3s^3\sqrt{\epsilon}}{2^{d+s}\l}\max(\|\vh\|_2\|\vm\|_2,\|\htrue\|_2\|\mtrue\|_2)\|\vm\|_2, \text{ and}\\
	 \|\vshm-\tshm\|_2 &= \lim_{\delta\rightarrow0^{+}} \|\nabla_{\vm}f\left((\vh,\vx)+\delta\vw\right) -\tshmdw\|_2\\
	&\leq K\frac{d^3s^3\sqrt{\epsilon}}{2^{d+s}\l}\max(\|\vh\|_2\|\vm\|_2,\|\htrue\|_2\|\mtrue\|_2)\|\vh\|_2,
\end{align*}
where the inequalities follow from \eqref{eq:gradh_con} and \eqref{eq:gradx_con}. 

Note that the one-sided directional derivative of $f$ in the direction of $(\vx,\vy)\neq \vzero$ at $(\vh,\vy)$ is $D_{(\vx,\vy)}f(\vh,\vx) = \lim_{t\rightarrow 0^{+}}\frac{1}{t}\left(f((\vh,\vx)+t(\vx,\vy))-f(\vh,\vx)\right)$. Due to the continuity and piecewise linearity of the function
\[
\mathcal{G}(\vh,\vm) = \Lambdah\vh\odot\Lambdam\vm,
\]
 we have that for any $(\vh,\vm)\neq (\vzero,\vzero)$ and $(\vx,\vy)\neq \vzero$ that there exists a sequence $\{(\vh_n,\vm_{n})\}\rightarrow(\vh,\vm)$ such that $f$ is differentiable at each $(\vh_n,\vm_n)$ and $D_{(\vx,\vy)}f(\vh,\vm) = \lim_{n\rightarrow\infty}\nabla f(\vh_n,\vm_n)\cdot(\vx,\vy)$. Thus, as $\nabla f(\vh_n,\vm_n) = \left[\begin{array}{c}\vfhmn\\ \vshmn \end{array}\right]$,
 \begin{align*}
 &D_{-\gfhm}f(\vh,\vm) = \lim_{n\rightarrow\infty}\nabla f(\vh_n,\vm_n) \cdot \frac{-\gfhm}{\|\gfhm\|_2}=\frac{-1}{\|\gfhm\|_2} \lim_{n\rightarrow\infty}\vfhmn\cdot\vfhm,\\
  &D_{-\gshm}f(\vh,\vm) = \lim_{n\rightarrow\infty}\nabla f(\vh_n,\vm_n) \cdot \frac{-\gshm}{\|\gshm\|_2}=\frac{-1}{\|\gshm\|_2} \lim_{n\rightarrow\infty}\vshmn\cdot\vshm.
 \end{align*}
 Now, we write
 \begin{align*}
 &\vfhmn\cdot\vfhm\\
 =& \tfhmn\cdot\tfhm+(\vfhmn-\tfhmn)\cdot\tfhm+\tfhmn\cdot(\vfhm-\tfhm)\\
 &+(\vfhmn-\tfhmn)\cdot(\vfhm-\tfhm)\\
 \geq&\tfhmn\cdot\tfhm-\|\vfhmn-\tfhmn\|_2\|\tfhm\|_2-\|\tfhmn\|_2\|\vfhm-\tfhm\|_2\\
 &\|\vfhmn-\tfhmn\|_2\|\vfhm-\tfhm\|_2\\
 \geq&\tfhmn\cdot\tfhm-K\frac{d^3s^3\sqrt{\epsilon}}{2^{d+s}\l}\max(\|\vh_n\|_2\|\vm_n\|_2,\|\htrue\|_2\|\mtrue\|_2)\|\vm_n\|_2\|\tfhm\|_2\\
 &-K\frac{d^3s^3\sqrt{\epsilon}}{2^{d+s}\l}\max(\|\vh\|_2\|\vm\|_2,\|\htrue\|_2\|\mtrue\|_2)\|\vm\|_2\|\tfhmn\|_2\\
 &-\left(K\frac{d^3s^3\sqrt{\epsilon}}{2^{d+s}\l}\right)^2\max(\|\vh_n\|_2\|\vm_n\|_2,\|\htrue\|_2\|\mtrue\|_2)\\
 &\max(\|\vh\|_2\|\vm\|_2,\|\htrue\|_2\|\mtrue\|_2)\|\vm_n\|_2\|\vm\|_2.
 \end{align*}
 As $\tfhm$ is continuous in $(\vh,\vm)$ for all $(\vh,\vm)\notin \mathcal{K}$, we have for all $(\vh,\vm)\notin \mathcal{S}_{4Kd^3s^3\sqrt{\epsilon},(\htrue,\mtrue)}^{(1)}\cup\mathcal{K}$,
 \begin{align}
 	&\lim_{n\rightarrow\infty}\vfhmn\cdot\vfhm\notag\\
 	\geq& \|\tfhm\|_2^2-2K\frac{d^3s^3\sqrt{\epsilon}}{2^{d+s}\l}\max(\|\vh\|_2\|\vm\|_2,\|\htrue\|_2\|\mtrue\|_2)\|\vm\|_2\|\tfhm\|_2\notag\\
 &-\left(K\frac{d^3s^3\sqrt{\epsilon}}{2^{d+s}\l}\max(\|\vh\|_2\|\vm\|_2,\|\htrue\|_2\|\mtrue\|_2)\|\vm\|_2\right)^2\notag\\
 \geq& \frac{\|\tfhm\|_2}{2}\Big[\|\tfhm\|_2-4K\frac{d^3s^3\sqrt{\epsilon}}{2^{d+s}\l}\max(\|\vh\|_2\|\vm\|_2,\|\htrue\|_2\|\mtrue\|_2)\|\vm\|_2\Big]+\notag\\
 &\frac{1}{2}\Big[\|\tfhm\|_2^2-2\left(K\frac{d^3s^3\sqrt{\epsilon}}{2^{d+s}\l}\max(\|\vh\|_2\|\vm\|_2,\|\htrue\|_2\|\mtrue\|_2)\|\vm\|_2\right)^2\Big]\notag\\
 >&0\label{eq:directional_first}.
 \end{align} 
 Similarly, we have for all $(\vh,\vm)\notin \mathcal{S}_{4Kd^3s^3\sqrt{\epsilon},(\htrue,\mtrue)}^{(2)}\cup\mathcal{K}$,
 \begin{align}
 	&\lim_{n\rightarrow\infty}\vshmn\cdot\vshm>0\label{eq:directional_second} .
 \end{align} 
 So, for all $(\vh,\vm)\notin \left(\mathcal{S}_{4Kd^3s^3\sqrt{\epsilon},(\htrue,\mtrue)}^{(1)}\cap \mathcal{S}_{4Kd^3s^3\sqrt{\epsilon},(\htrue,\mtrue)}^{(2)} \right)\cup\mathcal{K}$, at least \eqref{eq:directional_first} or \eqref{eq:directional_second} holds. If \eqref{eq:directional_first} holds, then we have $D_{-\gfhm}f(\vh,\vm)<0$ and if \eqref{eq:directional_second} holds, then we have $D_{-\gshm}f(\vh,\vm)\\<0$

It remain to prove that for all $(\vh,\vm) \in \mathcal{K}$ and for all $(\vx,\vy) \in \R^{n\times p}$, $D_{(\vx,\vy)}f(\vh,\vm)\leq 0$. We first assume $\vh = \vzero$ and $\vm$ is arbitrary. Let
\begin{align*}
\htruetilde = \Lambdahodmf \htrue ,\quad\xtilde = \Lambdaxdmf\vx,\quad\mtilde=\Lambdamdmf\vm,\quad\mtruetilde=\Lambdamodmf\mtrue,
\end{align*}
$\thetaibark = g(\thetaimbark)$ for $g$ given in \eqref{eq:g}, $\thetaobarf = \angle(\vx,\htrue)$ and $\thetaobars = \angle(\vm,\mtrue)$. We compute
\begin{align*}
	&-D_{(\vx,\vy)}f(\vh,\vm)\cdot\|(\vx,\vy)\|_2\\
	&=\lim_{t\rightarrow0^+}\frac{f((\vh,\vm)+t(\vx,\vy))-f(\vh,\vm)}{t}\\
	&= \left<\diag\left(\Lambdam \vm\right)\Lambdax\vx,\diag\left(\Lambdamo \mtrue \right)\Lambdaho\htrue \right>\\
	& = \left<\xtilde,\left(\Wdmfm \right)^\intercal\diag\left(\Wdmfx\mtilde\odot\Wdmfxo \mtruetilde \right)\Wdmfho \htruetilde \right>\\
	& = \Big<\xtilde,\Big(\left(\Wdmfm \right)^\intercal\diag\left(\Wdmfx\mtilde\odot\Wdmfxo \mtruetilde \right)\Wdmfho\\
	&\quad- \frac{\alpha}{n}\mQ_{\xtilde,\htruetilde}\mtildet \mQ_{\mtilde,\mtruetilde}\mtruetilde \Big)\htruetilde \Big>+\frac{\alpha}{n}\mtildet \mQ_{\mtilde,\mtruetilde}\mtruetilde \cdot\xtildet \mQ_{\xtilde,\htruetilde}\htruetilde\\
	&\geq-\Big\|\left(\Wdmfm \right)^\intercal\diag\left(\Wdmfx\mtilde\odot\Wdmfxo \mtruetilde \right)\Wdmfho\\
	&\quad- \frac{\alpha}{n}\mQ_{\xtilde,\htruetilde}\mtildet \mQ_{\mtilde,\mtruetilde}\mtruetilde \Big\| \|\xtilde\|_2\|\htruetilde\|_2+\frac{\alpha}{4n}\left(\frac{(\pi -\thetadmfbarf)\cos\thetadmfbarf+\sin\thetadmfbarf}{\pi}\right)\\
	&\quad\left(\frac{(\pi -\thetadmfbars)\cos\thetadmfbars+\sin\thetadmfbars}{\pi}\right)\|\xtilde\|\htruetilde\|\|\mtilde\|\|\mtruetilde\|\\
	&\geq - \frac{4\epsilon}{n}\|\mtilde\|_2\|\mtruetilde\|_2\|\xtilde\|_2\|\htruetilde\|_2 + \frac{\alpha}{4n}\|\xtilde\|\htruetilde\|\|\mtilde\|\|\mtruetilde\|\cos\thetadbarf\cos\thetadbars\\
	&\geq \left(- \frac{4\epsilon}{n}+ \frac{\alpha}{4\pi^2n}\right)\|\xtilde\|\htruetilde\|\|\mtilde\|\|\mtruetilde\|.
\end{align*}
So, if $4\pi^2\epsilon/\alpha<1$, then $D_{(\vx,\vy)}f(\vh,\vm)\cdot\|(\vx,\vy)\|_2 \leq 0$ for all $(\vx,\vy)\in\R^{n\times p}$ and $(\vh,\vm) \in \{(\vh,\vm)\left|\vh = \vzero,\ \vm\in\R^{p}\right.\}$. Similarly, $D_{(\vx,\vy)}f(\vh,\vm)\cdot\|(\vx,\vy)\|_2 \leq 0$ for all $(\vx,\vy)\in\R^{n\times p}$ and $(\vh,\vm) \in \{(\vh,\vm)\left|\vh \in \R^{n},\ \vm=\vzero\right.\}$.

Let $\mathcal{S} = \mathcal{S}_{4Kd^3s^3\sqrt{\epsilon},(\htrue,\mtrue)}^{(1)}\cap \mathcal{S}_{4Kd^3s^3\sqrt{\epsilon},(\htrue,\mtrue)}^{(2)}$. The proof is finished by applying Lemma \ref{lem:control_zeros} and $38(d^5+s^5)\sqrt{4Kd^3s^3\sqrt{\epsilon}}/\alpha<1$ to get 
 \begin{align*}
 \mathcal{S}\subseteq &\mathcal{A}_{\tilde{K}\frac{d^3s^3\epsilon^{1/4}}{\alpha},(\htrue,\mtrue)}\cup \mathcal{A}_{\tilde{K}\frac{d^8s^3\epsilon^{1/4}}{\alpha},\left(-\rhof\htrue,\mtrue\right)}\cup \mathcal{A}_{\tilde{K}\frac{d^3s^8\epsilon^{1/4}}{\alpha},\left(\rhos\htrue,-\mtrue\right)}\\
 &\cup \mathcal{A}_{\tilde{K}\frac{d^8s^8\epsilon^{1/4}}{\alpha},\left(-\rhof\rhos\htrue,-\mtrue\right)},
\end{align*}
for some absolute constant $\tilde{K}$.
\end{proof}

\subsection{Concentration of terms in $\tilde{\vg}_{1,(\vh,\vm)}$ and $\tilde{\vg}_{2,(\vh,\vm)}$}

\begin{lemma}\label{lem:concentration_no_comp}
Fix $0<\epsilon<d^{-4}/(16\pi)^2$ and $d\geq 2$. Suppose that $\mW_i \in\R^{n_{i}\times n_{i-1}}$ satisfies the WDC with constant $\epsilon$ and $1$ for $i=1,\dots,d$. Define
\[
\tilde{t}_{\vp,\vq} = \frac{1}{2^d}\left[\left(\prod_{i=0}^{d-1}\frac{\pi-\thetaibar}{\pi}\right)\vq+\sum_{i=0}^{d-1}\frac{\sin\thetaibar}{\pi}\left(\prod_{j=i+1}^{d-1}\frac{\pi-\thetajbar}{\pi}\right)\frac{\|\vq\|_2}{\|\vp\|_2}\vp \right],
\]
where $\thetaibar = g(\thetaimbar)$ for $g$ given by \eqref{eq:g} and $\thetaobar = \angle(\vp,\vq)$. For all $\vp\neq 0$ and $\vq\neq 0$,
\begin{align}
	&\left\|\left(\prod_{i=d}^{1}\mW_{i,+,\vp}\right)^\intercal\left(\prod_{i=d}^{1}\mW_{i,+,\vq}\right)\vq -\tilde{t}_{\vp,\vq}\right\|_2\leq24\frac{d^3\sqrt{\epsilon}}{2^d}\|\vq\|_2.
\end{align}
\end{lemma}
We refer the readers to \cite{Hand2017DCS} for proof of Lemma \ref{lem:concentration_no_comp}. We now state a related Lemma.
\begin{lemma}\label{lem:main_concentration}
Fix $0<\epsilon<1/((d^4+s^4)16\pi)^2$, $d\geq 2$ and $s\geq 2$. Suppose that $\Wf_i \in\R^{\nhi\times \nhimf}$ for $i=1,\dots,d-1$ and $\Ws_i \in\R^{\pgi\times \pgimf}$ for $i=1,\dots,s-1$ satisfy the WDC with constant $\epsilon$ and $1$. Suppose $\Wf_d \in \R^{\l\times n_{d-1} }$ satisfy WDC with constants $\epsilon$ and $\alpha_1$, and $\Ws_s \in \R^{\l\times p_{s-1}}$ satisfy WDC with constants $\epsilon$ and $\alpha_2$. Also, suppose $\left(\Wf_d,\Ws_s\right)$ satisfy pair-WDC with constants $\epsilon$, $\alpha =\alpha_1\cdot\alpha_2$. Define
\[
\tilde{t}_{\vp,\vq}^{(k)} = \frac{1}{2^\ak}\left[\left(\prod_{i=0}^{\ak-1}\frac{\pi-\thetaibark}{\pi}\right)\vq+\sum_{i=0}^{\ak-1}\frac{\sin\thetaibark}{\pi}\left(\prod_{j=i+1}^{\ak-1}\frac{\pi-\thetajbark}{\pi}\right)\frac{\|\vq\|_2}{\|\vp\|_2}\vp \right],
\]
where $\thetaibark = g(\thetaimbark)$ for $g$ given by \eqref{eq:g}, $\thetaobark = \angle(\vp,\vq)$, $\af=d$, and $\as = s$. For all $\vh\neq0$, $\vx\neq0$, $\vm\neq0$ and $\vy\neq0$, 
\begin{align}
	&\left\|\left(\Lambdah \right)^\intercal\diag\left(\Lambdam \vm \odot\Lambday \vy \right)\Lambdax \vx-\frac{\alpha}{n}\left(\mtrans\tilde{\vt}_{\vm,\vy}^{(2)}\right) \tilde{\vt}_{\vh,\vx}^{(1)}\right\|_2\notag\\
	&\hskip9cm\leq\frac{208d^3s^3\sqrt{\epsilon}}{2^{d+s}\l}\|\vx\|_2\|\vm\|_2\|\vy\|_2,\label{eq:pairWDC_concen1} \\
	&\left\|\left(\Lambdam \right)^\intercal\diag\left(\Lambdah \vh \odot\Lambdax \vx \right)\Lambday \vy-\frac{\alpha}{n}\left(\htrans\tilde{\vt}_{\vh,\vx}^{(1)}\right) \tilde{\vt}_{\vm,\vy}^{(2)} \right\|_2\notag\\
	&\hskip9cm\leq\frac{208d^3s^3\sqrt{\epsilon}}{2^{d+s}\l}\|\vy\|_2\|\vh\|_2\|\vx\|_2\label{eq:pairWDC_concen2}.
\end{align}
\end{lemma}

\begin{proof}
We will prove \eqref{eq:pairWDC_concen1}. Proof of \eqref{eq:pairWDC_concen2} is identical to proof of \eqref{eq:pairWDC_concen1}. Define $\htrue = \vh$, $\xtrue=\vx$, $\mtrue = \vm$, $\ytrue = \vy$,
\begin{align*}
	\vh_d := \left(\prodWh\right)\vh &=\left(\Wdh\Wdmfh\dots\Wfh\right)\vh\\&=\Wdh\vh_{d-1}\\&=(\mW_d^{(1)})_{+,\vh_{d-1}}\vh_{d-1},
\end{align*}
and analogously $\vx_d = \left(\prodWx\right)\vx$, $\vm_s = \left(\prodWm\right)\vm$, and $\vy_s = \left(\prodWy\right)\vy$. By the WDC, we have for all $\vh\neq \vzero$, $\vm\neq \vzero$,
\begin{align}
	& \left\|\left(\Wf_{i}\right)_{+,\vh}^\intercal\left(\Wf_{i}\right)_{+,\vh}-\frac{1}{2}\mI_{\nhimf} \right\|\leq \epsilon \text{ for all } i = 1,\dots, d-1, \text{ and}\\
	&\left\|\left(\Ws_{i}\right)_{+,\vm}^\intercal\left(\Ws_{i}\right)_{+,\vm}-\frac{1}{2}\mI_{\pgimf}\right\|\leq\epsilon\text{ for all } i = 1,\dots, s-1.
\end{align}
In particular, $\left\|\left(\Wih\right)^\intercal\Wih -\frac{1}{2}\mI_{\nhimf}\right\|\leq \epsilon$ and $\left\|\left(\Wim\right)^\intercal\Wim -\frac{1}{2}\mI_{\pgimf}\right\|\leq\epsilon$.
and consequently,
\begin{align*}
&\frac{1}{2}-\epsilon\leq\left\|\Wih \right\|^2\leq\frac{1}{2}+\epsilon\\	
&\frac{1}{2}-\epsilon\leq\left\|\Wim \right\|^2\leq\frac{1}{2}+\epsilon.
\end{align*}
Hence,
\begin{align}\label{eq:sizeofLambdas}
\left\|\prodWhdmf \right\|\left\|\prodWxdmf\right\| \leq\frac{1}{2^{d-1}}\left(1+2\epsilon\right)^{d-1} = \frac{1}{2^{d-1}}e^{(d-1)\log(1+2\epsilon)}\leq \frac{1+4\epsilon(d-1)}{2^{d-1}},
\end{align}
where we used that $\log(1+z)\leq z$, $e^z\leq 1+2z$ for $z<1$, and $2(d-1)\epsilon\leq 1$. Similarly,
\begin{align}\label{eq:sizeofLambdasGm}
\left\|\prodWmsmf \right\|\left\|\prodWysmf \right\| \leq\frac{1+4\epsilon(s-1)}{2^{s-1}}.
\end{align}

Let
\begin{align*}
\htilde = \Lambdahdmf\vh,\quad\xtilde = \Lambdaxdmf\vx,\quad\mtilde=\Lambdamdmf\vm,\quad\ytilde=\Lambdaydmf\vy,
\end{align*}
and consider
\begin{align}
&\left\|\left(\Lambdah\right)^\intercal\diag\left(\Lambdam\vm \odot\Lambday\vy \right)\Lambdax\vx-\frac{\alpha}{\l}\left(\mtrans\tilde{t}_{\vm,\vy}^{(2)}\right) \tilde{t}_{\vh,\vx}^{(1)}\right\|_2\notag\\
\leq &\Big\|\left(\Lambdahdmf \right)^\intercal\Big(\left(\Wdh\right)^\intercal\diag\Big(\Wdm\mtilde\odot\Wdy\ytilde\Big)\Wdx-\frac{\alpha}{\l}\mQ_{\htilde,\xtilde}\notag\\
&\quad\mtildet \mQ_{\mtilde,\ytilde}\ytilde\Big)\Lambdaxdmf\vx\Big\|_2+\Big\|\frac{\alpha}{\l}\left(\Lambdahdmf \right)^\intercal\mQ_{\htilde,\xtilde}\mtildet\mQ_{\mtilde,\ytilde}\ytilde\Lambdaxdmf\vx\notag\\
&\quad-\frac{\alpha}{\l}\left(\mtrans\tilde{t}_{\vm,\vy}^{(2)}\right) \tilde{t}_{\vh,\vx}^{(1)}\Big\|_2\notag\\
\leq &\Big\|\left(\Lambdahdmf \right)^\intercal\Big(\left(\Wdh\right)^\intercal\diag\Big(\Wdm\mtilde\odot\Wdy\ytilde\Big)\Wdx-\frac{\alpha}{\l}\mQ_{\htilde,\xtilde}\notag\\
&\quad\mtildet \mQ_{\mtilde,\ytilde}\ytilde\Big)\Lambdaxdmf\vx\Big\|_2\notag\\
&\quad+\frac{\alpha}{\l}\Big\|\mtildet\mQ_{\mtilde,\ytilde}\ytilde\left(\Lambdahdmf \right)^\intercal\mQ_{\htilde,\xtilde}\Lambdaxdmf\vx-\mtrans \tilde{t}_{\vm,\vy}^{(2)}\left(\Lambdahdmf \right)^\intercal\mQ_{\htilde,\xtilde}\Lambdaxdmf\vx\Big\|_2\notag\\
&\quad+\frac{\alpha}{\l}\Big\|\mtrans\tilde{t}_{\vm,\vy}^{(2)}\left(\Lambdahdmf \right)^\intercal\mQ_{\htilde,\xtilde}\Lambdaxdmf\vx-\mtrans\tilde{t}_{\vm,\vy}^{(2)}\tilde{t}_{\vh,\vx}^{(1)}\Big\|_2\label{eq:eqn1}
\end{align}
where both the first and second inequality holds because of triangle inequality. We bound the terms in the inequality above separately. First consider
\begin{align}
&\Big\|\left(\Lambdahdmf \right)^\intercal\Big(\left(\Wdh\right)^\intercal\diag\Big(\Wdm\mtilde\odot\Wdy\ytilde\Big)\Wdx-\frac{\alpha}{\l}\mQ_{\htilde,\xtilde}\notag\\
&\quad\mtildet \mQ_{\mtilde,\ytilde}\ytilde\Big)\Lambdaxdmf\vx\Big\|_2\notag\\
\leq&\Big\|\left(\Wdh\right)^\intercal\diag\Big(\Wdm\mtilde \odot\Wdy\ytilde\Big)\Wdx-\frac{\alpha}{\l}\mQ_{\htilde,\xtilde}\mtildet\mQ_{\mtilde,\ytilde}\ytilde\Big)\Big\|\notag\\
&\quad\left\|\Lambdahdmf\right\|\left\|\Lambdaxdmf\right\|\left\|\vx\right\|_2\notag\\
\leq &\left(\frac{1+4\epsilon(d-1)}{2^{d-1}}\right)\frac{4\epsilon}{\l}\|\vx\|_2\|\mtilde\|_2\|\ytilde\|_2\notag\\
= &\frac{\left(1+4\epsilon(d-1)\right)}{2^{d}}\frac{8\epsilon}{\l}\|\vx\|_2\|\Lambdamdmf\vm \|_2\|\Lambdaydmf\vy \|_2\notag\\
\leq &\frac{\left(1+4\epsilon(d-1)\right)}{2^{d}}\frac{\left(1+4\epsilon(s-1)\right)}{2^{s}}\frac{16\epsilon}{\l}\|\vx\|_2\|\vm \|_2\|\vy \|_2\notag\\
\leq&\frac{64\epsilon}{2^{d+s}\l}\|\vx\|_2\|\vm\|_2\|\vy\|_2.\label{eq:eqn2}
\end{align}
where the first inequality holds because spectral norm is a sub-multiplicative norm. The second inequality holds because of \eqref{eq:sizeofLambdas} and joint-WDC. The last inequality holds if $4\epsilon(d-1)<1$ and $4\epsilon(s-1)<1$.

Second, consider
\begin{align}
&\Big\|\mtildet\mQ_{\mtilde,\ytilde}\ytilde\left(\Lambdahdmf \right)^\intercal\mQ_{\htilde,\xtilde}\Lambdaxdmf\vx-\mtrans \tilde{t}_{\vm,\vy}^{(2)}\left(\Lambdahdmf \right)^\intercal\mQ_{\htilde,\xtilde}\Lambdaxdmf\vx\Big\|_2\notag\\
=&\Big\|\Big(\mtildet\mQ_{\mtilde,\ytilde}\ytilde-\mtrans \tilde{t}_{\vm,\vy}^{(2)}\Big)\left(\Lambdahdmf \right)^\intercal\mQ_{\htilde,\xtilde}\Lambdaxdmf\vx\Big\|_2\notag\\
\leq&\frac{1+4\epsilon(d-1)}{2^{d-1}}\|\mQ_{\htilde,\xtilde}\|\Big\|\left(\Lambdamdmf\right)^\intercal\mQ_{\mtilde,\ytilde}\Lambdaydmf\vy-\tilde{t}_{\vm,\vy}^{(2)}\Big\|_2\|\vx\|_2\|\|\vm\|_2\notag\\
=&\frac{1+4\epsilon(d-1)}{2^{d}}\Big\|\left(\Lambdamdmf\right)^\intercal \left(\mQ_{\mtilde,\ytilde}-\frac{1}{\alpha_2}\left(\Wdm\right)^\intercal\Wdy\right)\Lambdaydmf\vy\notag\\
&\quad+\frac{1}{\alpha_2}\left(\Lambdam\right)^\intercal\Lambday\vy -\tilde{t}_{\vm,\vy}^{(2)}\Big\|_2\|\vx\|_2\|\vm\|_2\notag\\
\leq&\frac{1+4\epsilon(d-1)}{2^{d}}\Big(\Big\|\Lambdamdmf\Big\| \Big\|\Lambdaydmf\Big\| \Big\|\left(\mQ_{\mtilde,\ytilde}-\frac{1}{\alpha_2}\left(\Wdm\right)^\intercal\Wdy\right)\Big\|\|\vy\|_2\notag\\
&\quad+\Big\|\frac{1}{\alpha_2}\left(\Lambdam\right)^\intercal\Lambday\vy -\tilde{t}_{\vm,\vy}^{(2)}\Big\|_2\Big)\|\vx\|_2\|\vm\|_2\notag\\
\leq&\frac{1+4\epsilon(d-1)}{2^{d+s}}\Big(6(1+4\epsilon(s-1))\epsilon/\alpha_2+24s^3\sqrt{\epsilon/\alpha_2}\Big)\|\vx\|_2\|\vm\|_2\|\vy\|_2\notag\\
\leq&\frac{2}{2^{d+s}}\Big(12\epsilon/\alpha_2+24s^3\sqrt{\epsilon/\alpha_2}\Big)\|\vx\|_2\|\vm\|_2\|\vy\|_2\notag\\
\leq&\frac{72s^3\sqrt{\epsilon}}{2^{2d}\alpha_2}\|\vx\|_2\|\vm\|_2\|\vy\|_2.\label{eq:eqn3}
\end{align}
where the first inequality holds because of \eqref{eq:sizeofLambdas}. The second inequality holds because of triangle inequality. The third inequality holds because of \eqref{eq:sizeofLambdasGm}, $\frac{1}{\sqrt{\alpha_2}}\Ws_d$ satisfy WDC with constant $\epsilon/\alpha_2$ and 1, and Lemma \ref{lem:concentration_no_comp}.

Third, consider
\begin{flalign}
&\Big\|\mtrans\tilde{t}_{\vm,\vy}^{(2)}\left(\Lambdahdmf \right)^\intercal\mQ_{\htilde,\xtilde}\Lambdaxdmf\vx-\mtrans\tilde{t}_{\vm,\vy}^{(2)}\tilde{t}_{\vh,\vx}^{(1)}\Big\|_2\notag\\
=&|\mtrans\tilde{t}_{\vm,\vy}^{(2)}|\Big\|\left(\Lambdahdmf \right)^\intercal\mQ_{\htilde,\xtilde}\Lambdaxdmf\vx-\tilde{t}_{\vh,\vx}^{(1)}\Big\|_2\notag\\
\leq&\|\tilde{t}_{\vm,\vy}^{(2)}\|_2\Big\|\left(\Lambdahdmf \right)^\intercal\left(\mQ_{\htilde,\xtilde}-\frac{1}{\alpha_1}\left(\Wdh\right)^\intercal\Wdx \right)\Lambdaxdmf\vx\notag\\
&\quad+\frac{1}{\alpha_1}\left(\Lambdah\right)^\intercal\Lambdax -\tilde{t}_{\vh,\vx}^{(1)}\Big\|_2\|\vm\|_2\notag\\
\leq&\frac{1+s}{2^s}\Big(\Big\|\Lambdahdmf\Big\|\Big\|\Lambdaxdmf\Big\|\Big\|\mQ_{\htilde,\xtilde}-\frac{1}{\alpha_1}\left(\Wdh\right)^\intercal\Wdx\Big\|\|\vx\|_2\notag\\
&\quad+\Big\|\frac{1}{\alpha_1}\left(\Lambdah\right)^\intercal\Lambdax\vx -\tilde{t}_{\vh,\vx}^{(1)}\Big\|_2\Big)\|\vm\|_2\|\vy\|_2\notag\\
\leq&\frac{2s}{2^{d+s}}\Big(6(1+4\epsilon(d-1))\epsilon/\alpha_1+24d^3\sqrt{\epsilon/\alpha_1}\Big)\|\vx\|_2\|\vm\|_2\|\vy\|_2\notag\\
\leq&\frac{72sd^3\sqrt{\epsilon}}{2^{d+s}\alpha_1}\|\vx\|_2\|\vm\|_2\|\vy\|_2.\label{eq:eqn4}
\end{flalign}
where the first inequality holds because of Cauchy-Schwartz inequality. The second inequality holds because of triangle inequality along with $\|\tilde{t}_{\vm,\vy}^{(2)}\|_2 \leq \frac{1+s}{2^s}$. The third inequality holds because of \eqref{eq:sizeofLambdas}, $\frac{1}{\sqrt{\alpha_1}}\Wf_d$ satisfy WDC with constant $\epsilon/\alpha_1$ and $1$, and Lemma \ref{lem:concentration_no_comp}.

Hence, combining \eqref{eq:eqn1}, \eqref{eq:eqn2}, \eqref{eq:eqn3}, and \eqref{eq:eqn4}, we get
\begin{flalign*}
&\left\|\left(\Lambdah\right)^\intercal\diag\left(\Lambdam\vm \odot\Lambday\vy \right)\Lambdax\vx-\frac{\alpha_1\alpha_2}{\l}\left(\mtrans\tilde{t}_{\vm,\vy}^{(2)}\right) \tilde{t}_{\vh,\vx}^{(1)}\right\|_2\\
\leq &\Big(\frac{64\epsilon}{2^{d+s}\l}+\frac{72s^3\sqrt{\epsilon}}{2^{d+s}\l}+\frac{72sd^3\sqrt{\epsilon}}{2^{2d}\l}\Big)\|\vx\|_2\|\vm\|_2\|\vy\|_2\\
\leq &\frac{208s^3d^3\sqrt{\epsilon}}{2^{d+s}\l}\|\vx\|_2\|\vm\|_2\|\vy\|_2.
\end{flalign*}
\end{proof}	


\subsection{Zeros of $\thmhomo$}

\begin{lemma}\label{lem:control_zeros}
Fix $0<\epsilon<1$ and $0<\alpha\leq 1$ such that $38(d^5+s^5)\sqrt{\epsilon}/\alpha < 1$. Let $\mathcal{K}=\{(\vh,\vzero)\in\R^{n\times p}\left| \vh\in\R^{n}\right.\}\cup\{(\vzero,\vm)\in\R^{n\times p}\big| \vm\in\R^{p}\big.\}$. Let 
\begin{align*}
&S_{\epsilon,(\htrue,\mtrue)}^{(1)} = \left\{(\vh,\vx)\in\R^{n\times p}\backslash\mathcal{K} \left|\frac{\|\tfhmhomo\|_2}{\|\vm\|_2 } \leq\frac{\epsilon\max(\|\vh\|_2\|\vm\|_2,\|\htrue\|_2\|\mtrue\|_2)}{2^{d+s}\l}\right.\right\},\\
&S_{\epsilon,(\htrue,\mtrue)}^{(2)} = \left\{(\vh,\vx)\in\R^{n\times p}\backslash\mathcal{K} \left|\frac{\|\tshmhomo\|_2}{\|\vh,\|_2 } \leq\frac{\epsilon\max(\|\vh\|_2\|\vm\|_2,\|\htrue\|_2\|\mtrue\|_2)}{2^{d+s}\l}\right.\right\},
&
\end{align*}
where $d$ and $s$ are an integers greater than 1. Let
\begin{equation}
	\tilde{\vt}_{\vm,\vy}^{(k)} = \frac{1}{2^\ak}\left(\prod_{i=0}^{\ak-1}\frac{\pi-\thetaibark}{\pi}\vy+\sum_{i=0}^{\ak-1}\frac{\sin\thetaibark}{\pi}\left(\prod_{j=i+1}^{\ak-1}\frac{\pi-\thetajbark}{\pi}\right)\frac{\|\vy\|_2}{\|\vm\|_2}\vm\right),
\end{equation}
where $\thetaibark = g(\thetaimbark)$ for $g$ given in \eqref{eq:g}, $\thetaobark = \angle(\vm,\vy)$, $\af = d$, and $\as = s$. Let $\thmhomo =\left[\begin{array}{c}\tfhmhomo\\\tshmhomo\end{array}\right]$ where
\begin{align}
	&\tfhmhomo =\frac{\alpha}{2^{d+s}\l}\|\vm\|_2^2\vh-\frac{\alpha}{\l}\mtrans\tstilde_{\vm,\mtrue}\tftilde_{\vh,\htrue},\\
	&\tshmhomo =\frac{\alpha}{2^{d+s}\l}\|\vh\|_2^2\vm-\frac{\alpha}{\l}\htrans\tftilde_{\vh,\htrue}\tstilde_{\vm,\mtrue}.
\end{align}
Define
\begin{equation*}
	\rhok:=\sum_{i=1}^{\ak-1}\frac{\sin\thetaichek}{\pi}\left(\prod_{j=i+1}^{\ak-1}\frac{\pi-\thetajchek}{\pi}\right),
\end{equation*}
where $\thetaochek = \pi$ and $\thetaichek = g(\thetaimchek)$. If $(\vh,\vm)\in S_{\epsilon,(\htrue,\mtrue)}^{(1)}\cap S_{\epsilon,(\htrue,\mtrue)}^{(2)}$  then one of the following holds:
\begin{align*}
	\bullet &\left|\thetaobarf\right|\leq 2\sqrt{\epsilon},\ \left|\thetaobars\right|\leq 2\sqrt{\epsilon} \text{ and }\\
	&\hskip1.5in \left|\|\vh\|_2\|\vm\|_2-\|\htrue\|_2\|\mtrue\|_2\right|\leq 145\frac{ds\sqrt{\epsilon}}{\alpha}\|\htrue\|_2\|\mtrue\|_2,\\
	\bullet & \left|\thetaobarf-\pi\right|\leq 12\pi^2d^3\sqrt{\epsilon}/\alpha,\ \left|\thetaobars\right|\leq 1.5\sqrt{\epsilon} \text{ and }\\
	&\hskip2in \left|\|\vh\|_2\|\vm\|_2-\rhof\|\htrue\|_2\|\mtrue\|_2\right|\leq 532\frac{d^6s\sqrt{\epsilon}}{\alpha}\|\htrue\|_2\|\mtrue\|_2,\\
	\bullet & \left|\thetaobarf\right|\leq 2\sqrt{\epsilon},\ \left|\thetaobars-\pi\right|\leq 12\pi^2s^3\sqrt{\epsilon}/\alpha \text{ and }\\
	&\hskip1.5in \left|\|\vh\|_2\|\vm\|_2-\rhof\|\htrue\|_2\|\mtrue\|_2\right|\leq 532\frac{ds^6\sqrt{\epsilon}}{\alpha}\|\htrue\|_2\|\mtrue\|_2,\\
	\bullet & \left|\thetaobarf-\pi\right|\leq 12\pi^2d^3\sqrt{\epsilon}\alpha,\ \left|\thetaobars-\pi\right|\leq 12\pi^2d^3\sqrt{\epsilon}/\alpha \text{ and }\\
	&\hskip1.5in \left|\|\vh\|_2\|\vm\|_2-\rhof\rhos\|\htrue\|_2\|\mtrue\|_2\right|\leq 3915 \frac{d^6s^6\sqrt{\epsilon}}{\alpha}\|\htrue\|_2\|\mtrue\|_2.
\end{align*}
In particular,
\begin{align*}
	S_{\epsilon,(\htrue,\mtrue)}^{(1)}\cap S_{\epsilon,(\htrue,\mtrue)}^{(2)}&\subseteq \mathcal{A}_{437\frac{ds\sqrt{\epsilon}}{\alpha},(\htrue,\mtrue)}\cup \mathcal{A}_{2436\pi \frac{d^6s\sqrt{\epsilon}}{\alpha},\left(-\rhof\htrue,\mtrue\right)}\\
	&\quad\cup \mathcal{A}_{2436\pi \frac{ds^6\sqrt{\epsilon}}{\alpha},\left(\rhos\htrue,-\mtrue\right)}\cup \mathcal{A}_{16767\pi^2 \frac{d^6s^6\sqrt{\epsilon}}{\alpha},\left(-\rhof\rhos\htrue,-\mtrue\right)},
\end{align*}
where $\mathcal{A}_{\epsilon,(\htrue,\mtrue)}$ is defined in \eqref{eq:recovery_area_hyper}.

\end{lemma}
\begin{proof}
Without loss of generality, let $\htrue = \ve_1$, $\mtrue = \ve_1$, $\hath = \cos\thetaobarf +\sin\thetaobarf$ and $\hatm = \cos\thetaobars+\sin\thetaobars$ for some $\thetaibarf,\thetaobars \in[0,\pi]$. First we introduce some notation for convenience. Let
\[
\xik = \prod_{i=0}^{\ak-1}\frac{\pi-\thetaibark}{\pi},\quad\zetak = \sum_{i=1}^{\ak-1}\frac{\sin\thetaibark}{\pi}\prod_{j=i+1}^{\ak-1}\frac{\pi-\thetajbark}{\pi},
\]
 $\rf = \|\vh\|_2,\ \rs = \|\vm\|_2,$ and $M = \max(\rf\rs,1)$. Using these notation, we can rewrite $\tfhmhomo$ as
 \begin{align*}
 	&\tfhmhomo\\
 	 = &\frac{\alpha}{2^{d+s}\l}\left(\|\vm\|_2\vh-\left(\xis\cos\thetaobars+\zetas\right)\left(\xif\frac{\htrue}{\|\htrue\|_2}+\zetaf\frac{\vh}{\|\vh\|_2}\right)\|\htrue\|_2\|\mtrue\|_2\right)\|\vm\|_2\\
	=& \frac{\alpha}{2^{d+s}\l}\left(\|\vm\|_2\vh-\cos\thetasbars \left(\xif\frac{\htrue}{\|\htrue\|_2}+\zetaf\frac{\vh}{\|\vh\|_2}\right)\|\htrue\|_2\|\mtrue\|_2\right)\|\vm\|_2\\
	=&\frac{\alpha}{2^{d+s}\l}\bigg(\rf\rs\left(\cos\thetaobarf\ve_1+\sin\thetaobarf\ve_2)\right)\\
	&\quad-\cos\thetasbars\Big(\xif\ve_1+\zetaf\left(\cos\thetaobarf\ve_1+\sin\thetaobarf\ve_2\right)\Big)\bigg)\rs.
 \end{align*}
 By inspecting the components of $\tfhmhomo$, we have that $(\vh,\vm)\in S_{\epsilon,(\htrue,\mtrue)}^{(1)}$ implies
 \begin{align}
 &\left|\rf\rs\cos\thetaobarf-\cos\thetasbars \left(\xif+\zetaf\cos\thetaobarf\right)\right|	\leq\frac{\epsilon M}{\alpha}\label{eq:tf_comp1}\\
  &\left|\rf\rs\sin\thetaobarf-\cos\thetasbars\zetaf\sin\thetaobarf\right|	\leq\frac{\epsilon M}{\alpha}\label{eq:tf_comp2}
 \end{align}
Similarly, by inspecting the components of $\tshmhomo$, we have that $(\vh,\vm)\in S_{\epsilon,(\htrue,\mtrue)}^{(2)}$ implies
 \begin{align}
 &\left|\rf\rs\cos\thetaobars-\cos\thetadbarf \left(\xis+\zetas\cos\thetaobars\right)\right|	\leq\frac{\epsilon M}{\alpha}\label{eq:ts_comp1}\\
  &\left|\rf\rs\sin\thetaobars-\cos\thetadbarf\zetas\sin\thetaobars\right|	\leq\frac{\epsilon M}{\alpha}\label{eq:ts_comp2}
 \end{align} 
Now, we record several properties. We have:
\begin{align}
	&\thetaibark \in[0,\pi/2]\text{ for } i\geq 1\\
	&\thetaibark\leq\bar{\theta}_{i-1}^{(k)}\text{ for } i\geq 1\\
	&|\xik|\leq 1\label{eq:ubound_xik}\\
		&\thetaichek\leq\frac{3\pi}{i+3}\text{ for }i\geq0\\
	&\thetaichek\geq\frac{\pi}{i+1}\text{ for } i\geq 0\\
	&\xik = \prod_{i=1}^{\ak-1}\frac{\pi-\thetaibark}{\pi}\geq\frac{\pi-\thetaobark}{\pi}\ak^{-3}\label{eq:lboundxik}\\
	&\thetaobark=\pi+O_1(\delta)\Rightarrow\thetaibark=\thetaichek+O_1(i\delta)\\
	&\thetaobark=\pi+O_1(\delta)\Rightarrow|\xik|\leq\frac{\delta}{\pi}\label{eq:xipi}\\
	&\thetaobark=\pi+O_1(\delta)\Rightarrow\zetak=\rhodk+O_1(3\ak^3\delta)\text{ if }\frac{\ak^2\delta}{\pi}\leq1\label{eq:zetapi}\\
	&|\zetak|= |\xik \cos\thetaobark -\cos\thetaabark| \leq 2\label{eq:ubound_zetak}\\
	&\cos\thetaibark\geq \frac{1}{\pi}\text{ for } i\geq 2\label{eq:cosdangle}
\end{align}
For a proof of \eqref{eq:ubound_xik}-\eqref{eq:zetapi}, we refer the readers to Lemma 8 of \cite{Hand2017DCS}. Also, we note that \eqref{eq:cosdangle} follows directly from \eqref{eq:ubound_zetak}.

We first show that if $(\vh,\vm) \in S_{\epsilon,(\htrue,\mtrue)}^{(1)}$ then $\rf\rs \leq 6$, and thus $M\leq 6$. Suppose $\rf\rs>1$. At least one of the following holds:$|\sin\thetaobarf|\geq1/\sqrt{2}$ or $|\cos\thetaobarf|\geq1/\sqrt{2}$. If $|\sin\thetaobarf|\geq 1/\sqrt{2}$ then \eqref{eq:tf_comp2} implies that $\left|\rf\rs-\cos\thetasbars\zetaf\right|\leq \sqrt{2}\epsilon\rf\rs/\alpha$. Using \eqref{eq:ubound_zetak}, we get $\rf\rs\leq \frac{2}{1-\sqrt{2}\epsilon/\alpha}\leq 4$ if $\epsilon/\alpha<1/4$. If $|\cos\thetaobarf|\geq1/\sqrt{2}$, then \eqref{eq:tf_comp1} implies $\left|\rf\rs-\cos\thetasbars\zetaf\right|\leq \sqrt{2}\left(\epsilon \rf\rs/\alpha+|\xif|\right)$. Using \eqref{eq:ubound_xik}, \eqref{eq:ubound_zetak}, and $\epsilon/\alpha<1/4$, we get $\rf\rs\leq \frac{\sqrt{2}|\xik|+\cos\thetasbars\zetaf}{1-\sqrt{2}\epsilon/\alpha}\leq \frac{2+\sqrt{2}}{1-\sqrt{2}\epsilon/\alpha}\leq 6$. Thus, we have $(\vh,\vm)\in S_{\epsilon,(\htrue,\mtrue)}^{(1)}\Rightarrow \rf\rs\leq 6\Rightarrow M\leq 6$. Similarly, we have $(\vh,\vm)\in S_{\epsilon,(\htrue,\mtrue)}^{(2)}\Rightarrow \rf\rs\leq 6\Rightarrow M\leq 6$.

Next we establish that we only need to consider the small angle case and the large angle case (i.e. $\thetaobark \approx 0 \text{ or } \pi$) if $(\vh,\vm)\in S_{\epsilon,(\htrue,\mtrue)}^{(1)}\cap S_{\epsilon,(\htrue,\mtrue)}^{(2)}$. Exactly one of the following holds: $\left|\rf\rs-\cos\thetasbars\zetaf\right|\geq \sqrt{\epsilon}M/\alpha$ or $\left|\rf\rs-\cos\thetasbars\zetaf\right|< \sqrt{\epsilon}M/\alpha$. If $\left|\rf\rs-\cos\thetasbars\zetaf\right|\geq \sqrt{\epsilon}M/\alpha$, then by \eqref{eq:tf_comp2}, we have $|\sin\thetaobarf|\leq\sqrt{\epsilon}$. Hence $\thetaobarf = O_1(2\sqrt{\epsilon})$ or $\thetaobarf = \pi+O_1(2\sqrt{\epsilon})$, as $\epsilon <1$. If $\left|\rf\rs-\cos\thetasbars\zetaf\right|< \sqrt{\epsilon}M/\alpha$, then by \eqref{eq:tf_comp1} and \eqref{eq:cosdangle} we have $\left|\xif\right|\leq 2\pi\sqrt{\epsilon}M/\alpha$. Using \eqref{eq:lboundxik}, we get $\thetaobarf = \pi +O_1(2\pi^2d^3\sqrt{\epsilon}M/\alpha)$. Thus, we only need to consider the small angle case, $\thetaobarf = O_1(2\sqrt{\epsilon})$ and the large angle case $\thetaobarf = \pi +O_1(12\pi^2d^3\sqrt{\epsilon}/\alpha)$, where we have used $M\leq 6$. Similarly, we only need to consider the small angle case, $\thetaobars = O_1(2\sqrt{\epsilon})$ and the large angle case $\thetaobars = \pi +O_1(12\pi^2s^3\sqrt{\epsilon}/\alpha)$.

{\bf Case 1: $\boldsymbol{\thetaobarf \approx 0}$ and $\boldsymbol{\thetaobars \approx 0}$ }. Assume $\thetaobark = O_1(2\sqrt{\epsilon})$. As $\thetaibark\leq\thetaobark\leq 2\sqrt{\epsilon}$ for all $i$, we have $\xik \geq \left(1-\frac{2\sqrt{\epsilon}}{\pi}\right)^\ak = 1+O_1(\frac{4\ak\sqrt{\epsilon}}{\pi})$ provided $2\ak\sqrt{\epsilon}\leq 1/2$. By \eqref{eq:ubound_zetak}, we also have $\zetak =O_1(\frac{\ak}{\pi}2\sqrt{\epsilon}) = O_1(\ak\sqrt{\epsilon})$. By \eqref{eq:tf_comp1}, we have
\[
\left|\rf\rs\cos\thetaobarf-\left(\xis\cos\thetaobars+\zetas\right) \left(\xif+\zetaf\cos\thetaobarf\right)\right|	\leq\frac{\epsilon M}{\alpha}
\]
where we used $\cos\thetaabark = \xik\cos\thetaobark+\zetak$. As $\cos\thetaobark = 1+O_1((\thetaobark)^2/2) = 1+O_1(2\epsilon)$,
\begin{align*}
\xis\cos\thetaobars+\zetas = &1+O_1(8s\epsilon\sqrt{\epsilon}+4s\sqrt{\epsilon}+2\epsilon+s\sqrt{\epsilon})=1+O_1(15s\sqrt{\epsilon}),\\
\xif+\zetaf\cos\thetaobarf =& 1+O_1(4d\sqrt{\epsilon}+2d\epsilon\sqrt{\epsilon}+d\sqrt{\epsilon})=1+O_1(7d\sqrt{\epsilon}).
\end{align*}
Thus,
\begin{align}
\rf\rs = 1+O_1(12\epsilon+6\epsilon/\alpha+105ds\epsilon+7d\sqrt{\epsilon}+15s\sqrt{\epsilon}) = 1+O_1(145ds\sqrt{\epsilon}/\alpha).
\end{align}

We now show $(\vh,\vm)$ is close to $\left(c\htrue,\frac{1}{c}\mtrue\right)$, where $c=\frac{\|\mtrue\|_2}{\|\vm\|_2}$. Consider
\begin{align*}
	&\left\|\vh-\frac{\|\mtrue\|_2}{\|\vm\|_2}\htrue\right\|_2\\
	\leq& \frac{1}{\|\vm\|_2}\left(\left|\|\vh\|_2\|\vm\|_2-\|\mtrue\|_2\|\htrue\|_2\right|+\left(\|\htrue\|_2\|\mtrue\|_2+\left|\|\vh\|_2\|\vm\|_2-\|\mtrue\|_2\|\htrue\|_2\right|\right)\thetaobarf\right)\\
	\leq& \frac{1}{\|\vm\|_2}\left(145ds\sqrt{\epsilon}\|\htrue\|_2\|\mtrue\|_2/\alpha+\left(\|\htrue\|_2\|\mtrue\|_2+145ds\sqrt{\epsilon}\|\htrue\|_2\|\mtrue\|_2 /\alpha\right)2\sqrt{\epsilon}\right)\\
	\leq& 437\frac{ds\sqrt{\epsilon}}{\alpha}\frac{\|\htrue\|_2\|\mtrue\|_2}{\|\vm\|_2}.
\end{align*}
Similarly,
\begin{align*}
	\left\|\vm-\frac{\|\vm\|_2}{\|\mtrue\|_2}\mtrue\right\|_2\leq \left(\left|\|\vm\|_2-\|\vm\|_2\right|+\left(\|\vm\|_2+\left|\|\vm\|_2-\|\vm\|_2\right|\right)\thetaobars\right)\leq 2\sqrt{\epsilon}\|\vm\|_2.
\end{align*}
Hence,
\begin{align*}
	\left\|(\vh,\vm)-\left(c\htrue,\frac{1}{c}\mtrue\right)\right\|_2\leq 437\frac{ds\sqrt{\epsilon}}{\alpha}\left\|\left(c\htrue,\frac{1}{c}\mtrue\right)\right\|_2.
\end{align*}

{\bf Case 2: $\boldsymbol{\thetaobarf \approx \pi}$ and $\boldsymbol{\thetaobars \approx 0}$ } Assume $\thetaobarf = \pi +O_1(\delta)$ where $\delta = 12\pi^2d^3\sqrt{\epsilon}/\alpha$. By \eqref{eq:xipi} and \eqref{eq:zetapi}, we have $\xif = O_1(\delta/\pi)$, and we have $\zetaf = \rhof+O_1(3d^3\delta)$ if $38d^5\sqrt{\epsilon}/\alpha\leq 1$. Also, assume $\thetaobars = O_1(2\sqrt{\epsilon})$. As $\thetaibars\leq\thetaobars\leq 2\sqrt{\epsilon}$ for all $i$, we have $\xis \geq \left(1-\frac{2\sqrt{\epsilon}}{\pi}\right)^s = 1+O_1(\frac{4s\sqrt{\epsilon}}{\pi})$ provided $2s\sqrt{\epsilon}\leq 1/2$. By \eqref{eq:ubound_zetak}, we also have $\zetas =O_1(\frac{s}{\pi}2\sqrt{\epsilon}) = O_1(s\sqrt{\epsilon})$. By \eqref{eq:ts_comp1}, we have
\[
\left|\rf\rs\cos\thetaobars-\left(\xif\cos\thetaobarf+\zetaf\right) \left(\xis+\zetas\cos\thetaobars\right)\right|	\leq\frac{\epsilon M}{\alpha}
\]
where we used $\cos\thetaabark = \xik\cos\thetaobark+\zetak$. As $\cos\thetaobarf=-1+O((\thetaobarf-\pi)^2/2)=-1+O_1(\delta^2/2)$ provided $\delta<1$ and $\cos\thetaobars = 1+O_1((\thetaobars)^2/2) = 1+O_1(2\epsilon)$,
\begin{align*}
\xif\cos\thetaobarf+\zetaf = &\rhof +O_1(\frac{\delta^3}{2\pi}+\frac{\delta}{\pi}+3d^3\delta) =\rhof+ O_1(4\delta d^3),\\
\xis+\zetas\cos\thetaobars =& 1+O_1(4s\sqrt{\epsilon}+2s\epsilon\sqrt{\epsilon}+s\sqrt{\epsilon})=1+O_1(7s\sqrt{\epsilon}).
\end{align*}
Thus,
\begin{align*}
	\rf\rs =& \rhof+O_1(12\epsilon+6\epsilon/\alpha +4\delta d^3+7s\sqrt{\epsilon}+28d^3s\delta\sqrt{\epsilon})\\
		=& \rhof+O_1(30d\sqrt{\epsilon}/\alpha +  4\delta d^3 +28d^3s\sqrt{\epsilon} )\\
		=&\rhof +O_1(532d^6s\sqrt{\epsilon}/\alpha).
\end{align*}
where, in the second equality, we use $\delta<1$. We now show $(\vh,\vm)$ is close to $\left(-c\rhof\htrue,\frac{1}{c}\mtrue\right)$, where $c=\frac{\|\mtrue\|_2}{\|\vm\|_2}$. Consider
\begin{align*}
	&\left\|\vh+\frac{\|\mtrue\|_2}{\|\vm\|_2}\rhof\htrue\right\|_2\\
	\leq& \frac{1}{\|\vm\|_2}\Big(\left|\|\vh\|_2\|\vm\|_2-\rhof\|\htrue\|_2\|\mtrue\|_2\right|+\Big(\rhof\|\htrue\|_2\|\mtrue\|_2+\big|\|\vh\|_2\|\vm\|_2\\
	&\quad-\rhof\|\htrue\|_2\|\mtrue\|_2\big|\Big)\thetaobarf\Big)\\
	\leq& \frac{1}{\|\vm\|_2}\left(532d^6s\sqrt{\epsilon}\|\mtrue\|_2\|\htrue\|_2/\alpha+\left(2\|\htrue\|_2\|\mtrue\|_2+532d^6s\sqrt{\epsilon}\|\mtrue\|_2\|\htrue\|_2 /\alpha\right)119d^3\sqrt{\epsilon}/\alpha\right)\\
	\leq& \frac{1}{\|\vm\|_2}\left(532d^6s\sqrt{\epsilon}\|\mtrue\|_2\|\htrue\|_2/\alpha+\left(2\|\htrue\|_2\|\mtrue\|_2+14s\|\mtrue\|_2\|\htrue\|_2 \right)119d^3\sqrt{\epsilon}/\alpha\right)\\
	\leq& 2436\pi\frac{ d^6s\sqrt{\epsilon}}{\alpha}\rhof\frac{\|\htrue\|_2\|\mtrue\|_2}{\|\vm\|_2}.
\end{align*}
Similarly,
\begin{align*}
	\left\|\vm-\frac{\|\vm\|_2}{\|\mtrue\|_2}\mtrue\right\|_2\leq \left(\left|\|\vm\|_2-\|\vm\|_2\right|+\left(\|\vm\|_2+\left|\|\vm\|_2-\|\vm\|_2\right|\right)\thetaobars\right)\leq 2\sqrt{\epsilon}\|\vm\|_2.
\end{align*}
Hence,
\begin{align*}
	\left\|(\vh,\vm)-\left(-c\rhof\htrue,\frac{1}{c}\mtrue\right)\right\|_2\leq 2436\pi \frac{d^6s\sqrt{\epsilon}}{\alpha}\left\|\left(-c\rhof\htrue,\frac{1}{c}\mtrue\right)\right\|_2.
\end{align*}

{\bf Case 3: $\boldsymbol{\thetaobarf \approx 0}$ and $\boldsymbol{\thetaobars \approx \pi}$ }. The analysis is similar to case 2. Using \eqref{eq:tf_comp1}, we get
\begin{align*}
	\rf\rs = \rhos +O_1(532ds^6\sqrt{\epsilon}/\alpha).
\end{align*}

Again, similar to case 2, we can show $(\vh,\vm)$ is close to $\left(c\rhos\htrue,-\frac{1}{c}\mtrue\right)$, where $c=\frac{\|\vh\|_2}{\|\htrue\|_2}$. We get,
\begin{align*}
	\left\|(\vh,\vm)-\left(c\rhos\htrue,-\frac{1}{c}\mtrue\right)\right\|_2\leq 2436\pi \frac{ds^6\sqrt{\epsilon}}{\alpha}\left\|\left(c\rhos\htrue,-\frac{1}{c}\mtrue\right)\right\|_2.
\end{align*}

{\bf Case 4: $\boldsymbol{\thetaobarf \approx \pi}$ and $\boldsymbol{\thetaobars \approx \pi}$}. Assume $\thetaobark = \pi +O_1(\deltak)$ where $\deltak = 12\pi^2\ak^3\sqrt{\epsilon}/\alpha$. By \eqref{eq:xipi} and \eqref{eq:zetapi}, we have $\xik = O_1(\deltak/\pi)$, and we have $\zetak = \rhodk +O_1(3\ak^3\deltak)$ if $\frac{\ak^2\delta}{\pi}\leq 1$. By \eqref{eq:tf_comp1}, we have
\[
\left|\rf\rs\cos\thetaobarf-\left(\xis\cos\thetaobars+\zetas\right) \left(\xif+\zetaf\cos\thetaobarf\right)\right|	\leq\frac{\epsilon M}{\alpha}
\]
where we used $\cos\thetaabark = \xik\cos\thetaobark+\zetak$. As $\cos\thetaobark=-1+O((\thetaobark-\pi)^2/2)=-1+O_1((\deltak)^2/2)$,
\begin{align*}
\xis\cos\thetaobars+\zetas = &\rhos +O_1(\frac{(\deltas){^3}}{2\pi}+\frac{\deltas}{\pi}+3s^3\deltas) =\rhos+ O_1(4\deltas s^3),\\
\xif+\zetaf\cos\thetaobarf =& -\rhof+O_1(\frac{\deltaf}{\pi}+\frac{3}{2}d^3(\deltaf){^3}+3\deltaf d^3)=-\rhof +O_1(5\deltaf d^3).
\end{align*}
Thus,
\begin{align*}
\rf\rs =& \rhof\rhos +O_1(6\epsilon/\alpha + 4(\deltaf)^2+4\deltas s^3 + 5\deltaf d^3 + 20\deltaf\deltas d^3s^3)\\	
 =& \rhof\rhos + O_1(6\epsilon/\alpha + 4\deltaf+4\deltas s^3 + 5\deltaf d^3 + 20\deltas d^3s^3)\\
 =& \rhof\rhos + O_1(6\epsilon/\alpha + 3909d^6s^6\sqrt{\epsilon}/\alpha)\\
  =& \rhof\rhos + O_1(3915d^6s^6\sqrt{\epsilon}/\alpha),
\end{align*}
where, in the second equality, we used $\deltaf\leq\frac{\pi}{d^2}<1$. We now show $(\vh,\vm)$ is close to $\left(-c\rhof\rhos\htrue,-\frac{1}{c}\mtrue\right)$, where $c=\frac{\|\mtrue\|_2}{\|\vm\|_2}$. Consider
\begin{align*}
	&\left\|\vh+\frac{\|\mtrue\|_2}{\|\vm\|_2}\rhof\rhos\htrue\right\|_2\\
	\leq& \frac{1}{\|\vm\|_2}\Big(\left|\|\vh\|_2\|\vm\|_2-\rhof\rhos\|\mtrue\|_2\|\htrue\|_2\right|+\Big(\rhof\rhos\|\htrue\|_2\|\mtrue\|_2+\\
	&\quad\left|\|\vh\|_2\|\vm\|_2-\rhof\rhos\|\mtrue\|_2\|\htrue\|_2\right|\Big)\thetaobarf\Big)\\
	\leq& \frac{1}{\|\vm\|_2}\big(3915d^6s^6\sqrt{\epsilon}\|\htrue\|_2\|\mtrue\|_2/\alpha+\big(4\|\htrue\|_2\|\mtrue\|_2\\
	&\quad+3915d^6s^6\sqrt{\epsilon}\|\htrue\|_2\|\mtrue\|_2/\alpha \big)119d^3\sqrt{\epsilon}/\alpha\big)\\
	\leq& \frac{1}{\|\vm\|_2}\left(3915d^6s^6\sqrt{\epsilon}\|\htrue\|_2\|\mtrue\|_2/\alpha+\left(4\|\htrue\|_2\|\mtrue\|_2+104ds^6\|\htrue\|_2\|\mtrue\|_2 \right)119d^3\sqrt{\epsilon}/\alpha\right)\\
	\leq& 16767\pi^2\frac{d^6s^6\sqrt{\epsilon}}{\alpha}\rhof\rhos\frac{\|\htrue\|_2\|\mtrue\|_2}{\|\vm\|_2}.
\end{align*}
Similarly,
\begin{align*}
	\left\|\vm+\frac{\|\vm\|_2}{\|\mtrue\|_2}\mtrue\right\|_2\leq \left(\left|\|\vm\|_2-\|\vm\|_2\right|+\left(\|\vm\|_2+\left|\|\vm\|_2-\|\vm\|_2\right|\right)\thetaobars\right)\leq 119s^3\sqrt{\epsilon}\|\vm\|_2/\alpha.
\end{align*}
Hence,
\begin{align*}
	\left\|(\vh,\vm)-\left(-c\rhof\rhos\htrue,-\frac{1}{c}\mtrue\right)\right\|_2\leq 16767\pi^2\frac{d^6s^6\sqrt{\epsilon}}{\alpha}\left\|\left(c\rhof\rhos\htrue,\frac{1}{c}\mtrue\right)\right\|_2.
\end{align*}
\end{proof}

\subsection{Proof of WDC condition}
We first state a lemma that shows that the weight $\mW\in \R^{\l\times n}$ of a layer of a neural network layer with i.i.d. $\mathcal{N}(0,1/\l)$ entries satisfies the WDC with constant $\epsilon$ and $1$, and we refer the readers to \cite{Hand2017DCS} for a proof of the lemma.
\begin{lemma}\label{lem:WDC}
Fix $0<\epsilon<1$. Let $\mW \in \R^{\l\times n}$ have i.i.d. $\mathcal{N}(0,1/\l)$ entries. If $\l>cn\log n$, then with probability at least $1-8\l e^{-\gamma n}$, $\mW$ satisfies the WDC with constant $\epsilon$ and $1$. Here $c, \gamma^{-1}$ are constants that depend only polynomially on $\epsilon^{-1}$.	
\end{lemma}
We now state a lemma similar to Lemma \ref{lem:WDC} which applies to truncated random variable. The proof follows the proof of lemma \ref{lem:WDC} in \cite{Hand2017DCS}.
\begin{lemma}\label{lem:WDC_truncated}
Fix $0<\epsilon<1$. Let $\mW \in \R^{\l\times n}$ where $i$th row of $\mW$ satisfy $\vw_{i}^\intercal = \vw^{\intercal} \cdot \vone_{\|\vw\|_2\leq 3\sqrt{n/\l}}$ and $\vw \sim \mathcal{N}(\vzero,\tfrac{1}{\l}\mI_{n})$. If $\l>cn\log n$, then with probability at least $1-8ne^{-\gamma n}$, $\mW$ satisfies the WDC with constant $\epsilon$ and $\alpha$. Here $c, \gamma^{-1}$ are constants that depend only polynomially on $\epsilon^{-1}$ and
\begin{equation}
	\alpha= \frac{\Gamma\left(\frac{n+2}{2}\right)-\Gamma\left(\frac{n+2}{2},\frac{9n}{2}\right)	}{\Gamma\left(\frac{n+2}{2}\right)},\label{eq:gamma_trunc_WDC}
\end{equation}
where $\Gamma$ is the Gamma function.
\end{lemma}
The WDC condition with constant $\epsilon$ and $\alpha$ can be written as
\[\|\mW_{+,\vx}^\intercal\mW_{+,\vy} - \alpha\mQ_{\vx,\vy}\|\leq \epsilon
\]
for all nonzero $\vx, \vy \in \R^n$. We note that 
\[\mW_{+,\vx}^\intercal\mW_{+,\vy} = \sum_{i = 1}^{\l} \vone_{\wit\vx}\vone_{\wit\vy}\vw_{i}\wit\]
and it is not continuous in $\vx$ and $\vy$. So, we consider an arbitrarily good continuous approximation of $\mW_{+,\vx}^\intercal\mW_{+,\vy}$. Let
\begin{align*}
	t_{-\epsilon}(z) = \left\{
	\begin{array}{ll}
	0 & z\leq -\epsilon,\\
	1+\frac{z}{\epsilon} & -\epsilon\leq z\leq 0,\\
	1 & z\leq 0,	
	\end{array}
\right. \quad \text{and} \quad 
	t_{\epsilon}(z) = \left\{
	\begin{array}{ll}
	0 & z\leq 0,\\
	\frac{z}{\epsilon} & 0\leq z\leq \epsilon,\\
	1 & z\geq \epsilon.	
	\end{array}
\right.
\end{align*}
and define
\begin{align*}
H_{-\epsilon}(\vx\vy)&:=\sum_{i=1}^\l t_{-\epsilon}(\wit\vx) t_{-\epsilon}(\wit\vy)\vw_i\wit,\\
H_{\epsilon}(\vx,\vy)&:=\sum_{i=1}^\l t_{\epsilon}(\wit\vx) t_{\epsilon}(\wit\vy)\vw_i\wit.
\end{align*}
The proof of Lemma \ref{lem:WDC_truncated} follows from the follow two lemmas. We first provide an upper bound on the singular values of $H_{-\epsilon}(\vx,\vy)$.
\begin{lemma}\label{lem:WDC_trunc_pos}
Fix $0<\epsilon<1$. Let ${\mW} \in \R^{\l\times n}$ where $i$th row of $\mW$ satisfy $\wit = \vw^{\intercal} \cdot \vone_{\|\vw\|_2\leq 3\sqrt{n}}$ and $\vw \sim \mathcal{N}(\vzero,\mI_{n})$. If $\l>cn\log n$, then with probability at least $1-4ne^{-\gamma n}$,
\begin{align*}	
&\forall (\vx,\vy)\neq(\boldsymbol{0},\boldsymbol{0}),\quad H_{-\epsilon}(\vx,\vy)\preceq \alpha\l \mQ_{\vx,\vy}+3\l\epsilon I_{n}.
\end{align*}
Here, $c$ and $\gamma^{-1}$ are constants that depend only polynomially on $\epsilon^{-1}$ and $\alpha$ is 
\begin{equation}
	\alpha= \frac{\Gamma\left(\frac{n+2}{2}\right)-\Gamma\left(\frac{n+2}{2},\frac{9n}{2}\right)	}{\Gamma\left(\frac{n+2}{2}\right)},
\end{equation}
where $\Gamma$ is the Gamma function.
\end{lemma}
\begin{proof} First we bound $\mathbb{E}[H_{-\epsilon}(\vx, \vy)]$ for fixed $\vx,\vy\in \mathcal{S}^{n-1}$. Noting that $t_{-\epsilon}(z)\leq \vone_{z\geq-\epsilon}(z) = \vone_{z>0}(z)+\vone_{-\epsilon\leq z\leq 0}(z)$, we have
	\begin{align*}
		&\mathbb{E}\left[H_{-\epsilon}(\vx,\vy)\right]\\
		\preceq & \mathbb{E}\left[\sum_{i=1}^\l\ind_{\wit\vx\geq -\epsilon} \ind_{\wit \vy\geq -\epsilon}\vw_i\wit\right]\\
		=&\l\mathbb{E}\left[\ind_{\wit\vx\geq -\epsilon} \ind_{\wit \vy\geq -\epsilon}\vw_i\wit\right]\\
		=&\l\mathbb{E}\Big[\left(\ind_{\wit\vx\geq0}\ind_{\wit \vy\geq0}\vw_i\wit\right)\Big]+2\l\mathbb{E}\left[\left(\ind_{-\epsilon\leq\wit\vx\leq0}\vw_i\wit\right)\right].
	\end{align*}
We first note that $\mathbb{E}\left[\ind_{\wit\vx\geq0}\ind_{\wit \vy\geq0}\bi\bit\right] = \alpha \mQ_{\vx,\vy}$ where $\alpha$ satisfies $0.97<\alpha<1$. Also, we have $\mathbb{E}\left[\ind_{-\epsilon\leq\wit\vx\leq0}\vw_i\wit\right]\preceq\frac{\epsilon\alpha}{2}\mI_{n}$. Thus,
	\begin{align}
	\mathbb{E}\left[H_{-\epsilon}(\vx,\vy)\right] &\preceq \alpha\l\cdot\mtrans \mQ_{\vx,\vy}\vy+\epsilon\alpha \l\mI_{n}\notag\\
		&\preceq\alpha\l\cdot\mtrans \mQ_{\vx,\vy}+\epsilon \l\mI_{n}\label{eq:expectation_pos_WDC}
	\end{align}
	
Second, we show concentration of $H_{-\epsilon}(\vx,\vy)$ for fixed $\vx,\vy \in \mathcal{S}^{n -1}$. Let 
\[\vxi_i =\sqrt{t_{-\epsilon}(\wit \vx) t_{-\epsilon}(\wit \vy)\vw_i}.\]
 We have
\begin{align*}
	& H_{-\epsilon}(\vx,\vy) - \mathbb{E}\left[H_{-\epsilon}(\vx,\vy)\right]\\
	=& \sum_{i=1}^\l\Big(t_{-\epsilon}(\wit\vx) t_{-\epsilon}(\wit\vy)\vw_i\wit - \mathbb{E}\left[t_{-\epsilon}(\wit\vx) t_{-\epsilon}(\wit\vy)\vw_i\wit\right]\Big)\\
	& = \sum_{i=1}^\l\left(\vxi_i\vxi_i^\intercal -\mathbb{E}\left[\vxi_i\vxi_i^\intercal\right] \right).
\end{align*}
Note that $\vxi_i$ is  sub-Gaussian for all $i$ and that the sub-Gaussian norm of $\vxi_i$ is bounded from above by an absolute constant which we call $K$. By first part of Remark 5.40 in \cite{vershynin10in}, there exists a $c_K$ and $\gamma_K$ such that for all $t\geq 0$, with probability at least $1- 2e^{-\gamma_K t^2}$, 
\begin{equation*}
\|H_{-\epsilon}(\vx,\vy) - \mathbb{E}\left[H_{-\epsilon}(\vx,\vy)\right]\|\leq \max(\delta,\delta^2)\l, \quad \text{where } \delta = c_K \sqrt{\frac{n}{\l}}+\frac{t}{\sqrt{\l}}.
\end{equation*}
If $\l>(2c_K/\epsilon)^2 n, t = \epsilon\sqrt{\l}/2$, and $\epsilon<1$, we have 
\begin{equation}\label{eq:concentration_pos_WDC}
\|H_{-\epsilon}(\vx,\vy) - \mathbb{E}\left[H_{-\epsilon}(\vx,\vy)\right]\|\leq \epsilon \l
\end{equation}
with probability at least $1- 2e^{-\gamma_K\frac{\epsilon^2\l}{4}}$.

{Third, we bound the Lipschitz constant of $H_{-\epsilon}$. For $\xtilde, \ytilde \in \mathbb{R}^{n}$ we have}
\begin{align*}
	& H_{-\epsilon}(\vx,\vy)-H_{-\epsilon}(\xtilde,\ytilde)\\
	=& \sum_{i=1}^{\l}\Big[t_{-\epsilon}(\wit\vx) t_{-\epsilon}(\wit\vy)-t_{-\epsilon}(\wit\xtilde) t_{-\epsilon}(\wit\ytilde)\Big]\vw_i\wit\\
	=&\sum_{i=1}^\l\Big[t_{-\epsilon}(\wit\vx)\left( t_{-\epsilon}(\wit\vy)-t_{-\epsilon}(\wit\ytilde)\right)\\
	&\quad+t_{-\epsilon}(\wit\ytilde)\left( t_{-\epsilon}(\wit\vx)-t_{-\epsilon}(\wit\xtilde)\right)\Big]\vw_i\wit\\
	=&\mW^\intercal\Big[\diag\left(t_{-\epsilon}(\mW\vx)\right)\diag( (\mW\vy)_+-(\mW\ytilde)_+)\\
	&\quad +\diag\left(t_{-\epsilon}(\mW\ytilde)\right)\diag\left( (\mW\vx)_+-(\mW\xtilde)_+\right)\Big]\mW
\end{align*}
Thus,
\begin{align*}
	&\|H_{-\epsilon}(\vx,\vy)-H_{-\epsilon}(\xtilde,\ytilde)\|\\
	\leq &\|\mW\|^2\Big[\|t_{-\epsilon}(\mW\vy)-t_{-\epsilon}(\mW\ytilde)\|_\infty+\|t_{-\epsilon}(\mW\vx)-t_{-\epsilon}(\mW\xtilde)\|_\infty\Big]\\
	\leq &\|\mW\|^2\Bigg[\max_{i\in[\l]}\left|t_{-\epsilon}(\wit\vy)-t_{-\epsilon}(\wit\ytilde)\right|+\max_{i\in[\l]}\left|t_{-\epsilon}(\wit\vx)-t_{-\epsilon}(\wit\xtilde)\right|\Bigg]\\
	\leq &\|\mW\|^2\Bigg[\max_{i\in[\l]}\frac{1}{\epsilon}\left|\wit(\vy-\ytilde)\right|+\max_{i\in[\l]}\frac{1}{\epsilon}\left|\wit(\vx-\xtilde)\right|\Bigg]\\
	\leq &\|\mW\|^2\Bigg[\frac{1}{\epsilon}\max_{i\in[\l]}\|\vw_i\|_2\left\|\vy-\ytilde\right\|+\frac{1}{\epsilon}\max_{i\in[\l]}\|\vw_i\|_2\left\|\vx-\xtilde\right\|\Bigg]\\
	\leq & \|\mW\|^2\Big[\frac{9}{\epsilon}\sqrt{n}\left\|\vx-\xtilde\right\|+\frac{9}{\epsilon}\sqrt{n}\left\|\vh-\htilde\right\|\Big]
\end{align*}
where the first inequality follows because $|t_{-\epsilon}(z)|\leq 1$ for all $z$, and the third inequality follows because $t_{-\epsilon}(z)$ is $1/\epsilon$-Lipschitz. { Let $E_1$ be the event that $\|\mW\|\leq 3\sqrt{\l}$. By Corollary 5.35 in \cite{vershynin10in}, for $\mA \in \R^{\l\times n}$ with rows of $\mA$ following $\mathcal{N}(\vzero,\mI_{n})$, we have $\mathbb{P}(\|\mA\|\leq 3\sqrt{\l})\geq 1-2e^{-\l/2}$, if $\l\geq n$. As rows of $\mW$ are truncated, we have $\mathbb{P}(E_1)\geq 1-2e^{-\l/2}$, if $\l\geq n$ as well.} On $E_1$, we have
\begin{align}
	&\|H_{-\epsilon}(\vx,\vy)-H_{-\epsilon}(\xtilde,\ytilde)\|\notag\\
	\leq&\frac{27\l\sqrt{n}}{\epsilon}\left[\left\|\vx-\ytilde\right\|+\left\|\vy-\ytilde\right\|\right]\label{eq:Lipschitz_pos_WDC}
\end{align}
for all $\vx,\vy,\xtilde, \ytilde \in \mathcal{S}^{n-1}$.

Finally, we complete the proof by a covering argument. Let $\mathcal{N}_\delta$ be a $\delta$-net on $\mathcal{S}^{n-1}$ such that $|\mathcal{N}_\delta|\leq (3/\delta)^{n}$. Take $\delta = \frac{\epsilon^2}{54\sqrt{n}}$. Combining \eqref{eq:expectation_pos_WDC} and \eqref{eq:concentration_pos_WDC}, we have
\begin{align*}
	\forall (\vx,\vy),\in \mathcal{N}_\delta,\quad H_{-\epsilon}(\vx,\vy)\preceq&\mathbb{E}H_{-\epsilon}(\vx,\vy) + \l\epsilon I_{n}\\
	\preceq&\alpha \l \mQ_{\vx,\vy}+2\l\epsilon I_{n}{\change.}
\end{align*}
with probability at least
\begin{equation*}
1-2|\mathcal{N}_\delta|e^{-\gamma_K\epsilon^2\l/4}\geq 1-2\left(\frac{3}{\delta}\right)^{n}e^{-\gamma_K\epsilon^2\l/4}\geq 1-2e^{-\gamma_K\epsilon^2\l/4+n\log(3\cdot 54\sqrt{n}/\epsilon^2)}.
\end{equation*}
If $\l\geq \tilde{c}n\log(n)$ for some $\tilde{c} =\Omega(\epsilon^2\log \epsilon)$, then this probability is at least $1 - 2e^{-\tilde{\gamma}\l}$ for some $\tilde{\gamma} = O(\epsilon^2)$. For $\vx, \vy \in \mathcal{S}^{n-1}$, let $\xtilde,\ytilde \in \mathcal{N}_\delta$ be such that $\|\vx-\xtilde\|_2\leq\delta$, and $\|\vy-\ytilde\|_2\leq\delta$. By \eqref{eq:Lipschitz_pos_WDC}, we have that
\begin{align*}	
\forall \vx, \vy \neq \vzero,\quad  &H_{-\epsilon}(\vx,\vy)\\
\preceq & H_{-\epsilon}(\xtilde,\ytilde) +\frac{27\l\sqrt{n}}{\epsilon}2\delta \mI_{n}\\
\preceq & \alpha\l \mQ_{\vx,\vy}+3\l\epsilon \mI_{n}.
\end{align*}
In conclusion, the result of this lemma holds if $\l>(2c_K/\epsilon)^2n$ and $\l\geq \tilde{c}(n)\log n$, with probability at least $1-2e^{-\gamma_K\epsilon^2\l/4}-2e^{-\l/2}-2e^{-\tilde{\gamma}\l}>1-6e^{-\gamma \l}$ for some $\gamma = O(\epsilon^2)$ and $\tilde{c} = \Omega(\epsilon^2\log\epsilon)$.
\end{proof}

Next, we now provide an upper bound on the singular values of $G_{\epsilon}(\vh,\vx,\vm,\vy)$.
\begin{lemma}\label{lem:WDC_trunc_neg}
Fix $0<\epsilon<1$. Let ${\mW} \in \R^{\l\times n}$ where $i$th row of $\mW$ satisfy $\wit = \vw^{\intercal} \cdot \vone_{\|\vw\|_2\leq 3\sqrt{n}}$ and $\vw \sim \mathcal{N}(\vzero,\mI_{n})$. If $\l>cn\log n$, then with probability at least $1-4ne^{-\gamma n}$,
\begin{align*}	
&\forall (\vx,\vy)\neq(\boldsymbol{0},\boldsymbol{0}),\quad H_{\epsilon}(\vx,\vy)\succeq \alpha\l \mQ_{\vx,\vy}-3\l\epsilon I_{n}.
\end{align*}
Here, $c$ and $\gamma^{-1}$ are constants that depend only polynomially on $\epsilon^{-1}$ and $\alpha$ is 
\begin{equation}
	\alpha= \frac{\Gamma\left(\frac{n+2}{2}\right)-\Gamma\left(\frac{n+2}{2},\frac{9n}{2}\right)	}{\Gamma\left(\frac{n+2}{2}\right)},
\end{equation}
where $\Gamma$ is the Gamma function.
\end{lemma}
\begin{proof} First we bound $\mathbb{E}[H_{-\epsilon}(\vx, \vy)]$ for fixed $\vx,\vy\in \mathcal{S}^{n-1}$. Noting that $t_{\epsilon}(z)\geq \vone_{z>0}(z)-\vone_{-\epsilon\leq z\leq 0}(z)$ for all $z$, we have
	\begin{align*}
		&\mathbb{E}\left[H_{\epsilon}(\vx,\vy)\right]\\
		\succeq & \l\mathbb{E}\Big[\left(\ind_{\wit\vx\geq0}\ind_{\wit \vy\geq0}\vw_i\wit\right)\Big]-2\l\mathbb{E}\left[\left(\ind_{-\epsilon\leq\wit\vx\leq0}\vw_i\wit\right)\right].
	\end{align*}
We first note that $\mathbb{E}\left[\ind_{\wit\vx\geq0}\ind_{\wit \vy\geq0}\bi\bit\right] = \alpha \mQ_{\vx,\vy}$ where $\alpha$ satisfies $0.97<\alpha<1$. Also, we have $\mathbb{E}\left[\ind_{-\epsilon\leq\wit\vx\leq0}\vw_i\wit\right]\preceq\frac{\epsilon\alpha}{2}\mI_{n}$. Thus,
	\begin{align}
	\mathbb{E}\left[H_{-\epsilon}(\vx,\vy)\right] &\succeq \alpha\l\cdot\mtrans \mQ_{\vx,\vy}\vy-\epsilon\alpha \l\mI_{n}\notag\\
		&\succeq\alpha\l\cdot\mtrans \mQ_{\vx,\vy}-\epsilon \l\mI_{n}
	\end{align}

Second, the same argument as in Lemma \ref{lem:WDC_trunc_pos} provides that for fixed $\vx,\vy \in \mathcal{S}^{n -1}$, if $\l>(2c_K/\epsilon)^2 n$, then we have with probability at least $1- 2e^{-\gamma_K\frac{\epsilon^2\l}{4}}$,
\begin{equation}\label{eq:concentration_pos_WDC}
\|H_{-\epsilon}(\vx,\vy) - \mathbb{E}\left[H_{-\epsilon}(\vx,\vy)\right]\|\leq \epsilon \l
\end{equation}

Third, same argument as in Lemma \ref{lem:WDC_trunc_pos} provides on the event $E_1$, we have 
\begin{align*}
	\|H_{-\epsilon}(\vx,\vy)-H_{-\epsilon}(\xtilde,\ytilde)\|\leq&\frac{27\l\sqrt{n}}{\epsilon}\left[\left\|\vx-\ytilde\right\|+\left\|\vy-\ytilde\right\|\right]
\end{align*}
for all $\vx,\vy,\xtilde, \ytilde \in \mathcal{S}^{n-1}$.

Finally, we complete the proof by an identical covering argument  as in Lemma \ref{lem:WDC_trunc_pos}. We have if $\l\geq c_0n\log n$ then with probability at least $1- 6e^{-\gamma\l}$,
\begin{align*}	
\forall \vx, \vy \neq \vzero,\quad  H_{\epsilon}(\vx,\vy)\succeq  \alpha\l \mQ_{\vx,\vy}-3\l\epsilon \mI_{n}.
\end{align*}
\end{proof}

\subsection{Proof of joint-WDC condition}
We now state a result that states random gaussian matrices with truncated rows satisfy joint-WDC. 
\begin{lemma}\label{lem:joint_WDC_trunc}
Fix $0<\epsilon<1$. Let ${\mB} \in \R^{\l\times n}$ where $i$th row of $\mB$ satisfy $\bit = \vb^{\intercal} \cdot \vone_{\|\vb\|_2\leq 3\sqrt{n/\l}}$ and $\vb \sim \mathcal{N}(\vzero,\mI_{n}/\l)$. Similarly, let ${\mC} \in \R^{\l\times p}$ where $i$th row of $\mC$ satisfy $\cit = \vc^{\intercal} \cdot\boldsymbol{1}_{\|\vc\|_2\leq 3\sqrt{p/\l}}$ and $\vc \sim \mathcal{N}(\vzero,\mI_{p}/\l)$. If $\l>c((n\log n)^2+(p\log p)^2)$, then with probability at least $1-8e^{-\gamma \l}$, $\mB$ and $\mC$ satisfy joint-WDC with constants $\epsilon$ and $\alpha=\alpha_1\cdot\alpha_2$. Here, $c$ and $\gamma^{-1}$ are constants that depend only polynomially on $\epsilon^{-1}$ and 
\begin{equation}\label{eq:gamma_trunc}
\alpha_1 = \frac{\Gamma\left(\frac{n+2}{2}\right)-\Gamma\left(\frac{n+2}{2},\frac{9n}{2}\right)	}{\Gamma\left(\frac{n+2}{2}\right)} \text{ and } \alpha_2 = \frac{\Gamma\left(\frac{p+2}{2}\right)-\Gamma\left(\frac{p+2}{2},\frac{9p}{2}\right)	}{\Gamma\left(\frac{p+2}{2}\right)},
\end{equation}
where $\Gamma$ is the Gamma function.
\end{lemma}
The proof of Lemma \ref{lem:joint_WDC_trunc} follows directly from Lemmas \ref{lem:joint_WDC_trunc_pos} and \ref{lem:joing_WDC_trunc_neg}. Using Corollary \ref{cor:singular_noniso}, we provide a concentration result of $\Bh^{\intercal}\diag(\Cm\vm)\diag(\Cy\vy)\Bx$, which is part of the joint-WDC condition. We note that
\begin{align*}
\Bh^{\intercal}\diag(\Cm\vm)\diag(\Cy\vy)\Bx =\sum_{i=1}^\l\boldsymbol{1}_{\bit \vh>0}\boldsymbol{1}_{\bit \vx>0}(\cit \vm)_+ (\cit \vy)_+ \bi\bit
\end{align*}
and it is not continuous in $\vh$ and $\vx$. So, we consider an arbitrarily good continuous approximation of $\Bh^{\intercal}\diag(\Cm\vm)\diag(\Cy\vy)\Bx$. Let
\begin{align*}
	t_{-\epsilon}(z) = \left\{
	\begin{array}{ll}
	0 & z\leq -\epsilon,\\
	1+\frac{z}{\epsilon} & -\epsilon\leq z\leq 0,\\
	1 & z\leq 0,	
	\end{array}
\right. \quad \text{and} \quad 
	t_{\epsilon}(z) = \left\{
	\begin{array}{ll}
	0 & z\leq 0,\\
	\frac{z}{\epsilon} & 0\leq z\leq \epsilon,\\
	1 & z\geq \epsilon.	
	\end{array}
\right.
\end{align*}
and define
\begin{align*}
G_{-\epsilon}(\vh,\vx,\vm,\vy)&:=\sum_{i=1}^\l t_{-\epsilon}(\bit\vh) t_{-\epsilon}(\bit\vx)(\cit\vm)_+ (\cit\vy)_+ \bi\bit,\\
G_{\epsilon}(\vh,\vx,\vm,\vy)&:=\sum_{i=1}^\l t_{\epsilon}(\bit\vh) t_{\epsilon}(\bit\vx)(\cit\vm)_+ (\cit\vy)_+ \bi\bit.
\end{align*}
We now provide an upper bound on the singular values of $G_{-\epsilon}(\vh,\vx,\vm,\vy)$.
\begin{lemma}\label{lem:joint_WDC_trunc_pos}
Fix $0<\epsilon<1$. Let ${\mB} \in \R^{\l\times n}$ where $i$th row of $\mB$ satisfy $\bit = \vb^{\intercal} \cdot \vone_{\|\vb\|_2\leq 3\sqrt{n}}$ and $\vb \sim \mathcal{N}(\vzero,\mI_{n})$. Similarly, let ${\mC} \in \R^{\l\times p}$ where $i$th row of $\mC$ satisfy $\cit = \vc^{\intercal} \cdot\boldsymbol{1}_{\|\vc\|_2\leq 3\sqrt{p}}$ and $\vc \sim \mathcal{N}(\vzero,\mI_{p})$. If $\l>c((n\log n)^2+(p\log p)^2)$, then with probability at least $1-4e^{-\gamma \l}$,
\begin{align*}	
&\forall (\vh,\vx)\neq(\boldsymbol{0},\boldsymbol{0}) \text{ and } \vm, \vy \in \mathcal{S}^{p-1},\\
&\quad\qquad G_{-\epsilon}(\vh,\vx,\vm,\vy)\preceq \alpha_1\alpha_2\l \mQ_{\vh,\vx}\mtrans \mQ_{\vm,\vy}\vy+4\l\epsilon I_{n}.
\end{align*}
Here, $c$ and $\gamma^{-1}$ are constants that depend only polynomially on $\epsilon^{-1}$ and $\alpha_1$ and $\alpha_2$ is as in \eqref{eq:gamma_trunc}.
\end{lemma}
\begin{proof} First we bound $\mathbb{E}[G_{-\epsilon}(\vh,\vx,\vm,\vy)]$ for fixed $\vh,\vx\in \mathcal{S}^{n-1}$ and $\vm,\vy \in\mathcal{S}^{p-1}$. Noting that $t_{-\epsilon}(z)\leq \vone_{z\geq-\epsilon}(z) = \vone_{z>0}(z)+\vone_{-\epsilon\leq z\leq 0}(z)$, we have
	\begin{align*}
		&\mathbb{E}\left[G_{-\epsilon}(\vh,\vx,\vm,\vy)\right]\\
		\preceq & \mathbb{E}\left[\sum_{i=1}^\l\ind_{\bit\vh\geq -\epsilon} \ind_{\bit \vx\geq -\epsilon} (\cit \vm)_+ (\cit \vy)_+ \bi\bit\right]\\
		=&\l\mathbb{E}\left[\ind_{\bit\vh\geq -\epsilon} \ind_{\bit \vx\geq -\epsilon} (\cit \vm)_+ (\cit \vy)_+ \bi\bit\right]\\	
		\preceq &\l\mathbb{E}\Big[\Big(\ind_{\bit\vh\geq0}\ind_{\bit \vx\geq0}(\cit \vm)_+ (\cit \vy)_+ +\left(\ind_{-\epsilon\leq\bit\vh\leq0}+\ind_{-\epsilon\leq\bit\vx\leq0}\right)(\cit \vm)_+ (\cit \vy)_+\Big)\bi\bit\Big]\\
		=&\l\mathbb{E}\Big[\left(\ind_{\bit\vh\geq0}\ind_{\bit \vx\geq0}\bi\bit\right)(\cit \vm)_+ (\cit \vy)_+\Big]+2\l\mathbb{E}\left[\left(\ind_{-\epsilon\leq\bit\vh\leq0}\bi\bit\right)(\cit \vm)_+ (\cit \vy)_+\right].
	\end{align*}
We first note that $\mathbb{E}\left[\ind_{\bit\vh\geq0}\ind_{\bit \vx\geq0}\bi\bit\right] = \alpha_1 \mQ_{\vh,\vx}$ and $\mathbb{E}\left[(\cit\vm)_+(\cit\vy)_+\right] = \alpha_2 \mtrans \mQ_{\vm,\vy}\vy$ where $\alpha_i$ satisfies $0.97<\alpha_i<1$. Also, we have $\left|\mtrans \mQ_{\vm,\vy}\vy\right| \leq \frac{1}{2}$ and $\mathbb{E}\left[\ind_{-\epsilon\leq\bit\vh\leq0}\bi\bit\right]\preceq\frac{\epsilon\alpha_1}{2}\mI_{n}$. Thus,
	\begin{align}
	\mathbb{E}\left[G_{-\epsilon}(\vh,\vx,\vm,\vy)\right] &\preceq \alpha_1\alpha_2\l\cdot\mtrans \mQ_{\vm,\vy}\vy\cdot \mQ_{\vh,\vx}+2\l\cdot\mathbb{E}\left[\ind_{-\epsilon\leq\bit\vh\leq0}\bi\bit\right]\cdot\alpha_2\mtrans \mQ_{\vm,\vy}\vy\notag\\
		&\preceq\alpha_1\alpha_2 \l\cdot\mtrans \mQ_{\vm,\vy}\vy\cdot \mQ_{\vh,\vx}+\frac{\epsilon\alpha_1\alpha_2 \l}{2}\mI_{n}\notag\\
		&\preceq\alpha_1\alpha_2 \l\cdot\mtrans \mQ_{\vm,\vy}\vy\cdot \mQ_{\vh,\vx}+\frac{\epsilon \l}{2}\mI_{n}\label{eq:joint_expectation_pos}
	\end{align}
	
Second, we show concentration of $G_{-\epsilon}(\vh,\vx,\vm,\vy)$ for fixed $\vh, \vx \in \mathcal{S}^{n -1}$ and $\vm, \vy \in \mathcal{S}^{p-1}$. Let $\vxi_i =\sqrt{t_{-\epsilon}(\bit \vh) t_{-\epsilon}(\bit \vx)(\cit \vm)_+ (\cit \vy)_+}\bi$. We have
\begin{align*}
	& G_{-\epsilon}(\vh,\vx,\vm,\vy) - \mathbb{E}\left[G_{-\epsilon}(\vh,\vx,\vm,\vy)\right]\\
	=& \sum_{i=1}^\l\Big(t_{-\epsilon}(\bit\vh) t_{-\epsilon}(\bit\vx)(\cit\vm)_+ (\cit\vy)_+ \bi\bit - \mathbb{E}\left[t_{-\epsilon}(\bit\vh) t_{-\epsilon}(\bit\vx)(\cit\vm)_+ (\cit\vy)_+ \bi\bit\right]\Big)\\
	& = \sum_{i=1}^\l\left(\vxi_i\vxi_i^\intercal -\mathbb{E}\left[\vxi_i\vxi_i^\intercal\right] \right).
\end{align*}
Note that $\vxi_i$ is  sub-Gaussian for all $i$ and that the sub-Gaussian norm of $\vxi_i$ is bounded from above by $K = \tilde{K}\sqrt{n}$, where $\tilde{K}$ is an absolute constant. By Corollary \ref{cor:singular_noniso}, there exists a $c = \bar{c}\sqrt{n}\log n$ and $\gamma = \frac{\bar{\gamma}}{n\log n}$ such that for all $t\geq 0$, with probability at least $1- 2e^{-\gamma t^2}$, 
\begin{equation*}
\|G_{-\epsilon}(\vh,\vx,\vm,\vy) - \mathbb{E}\left[G_{-\epsilon}(\vh,\vx,\vm,\vy)\right]\|\leq \max(\delta,\delta^2)\l, \quad \text{where } \delta = c \sqrt{\frac{n}{\l}}+\frac{t}{\sqrt{\l}}.
\end{equation*}
Here, $\bar{c}$ and $\bar{\gamma}$ are absolute constants. If $\l>(2\bar{c}/\epsilon)^2 n^2(\log n)^2, t = \epsilon\sqrt{\l}/2$, and $\epsilon<1$, we have 
\begin{equation}\label{eq:joint_concentration_pos}
\|G_{-\epsilon}(\vh,\vx,\vm,\vy) - \mathbb{E}\left[G_{-\epsilon}(\vh,\vx,\vm,\vy)\right]\|\leq \epsilon \l
\end{equation}
with probability at least $1- 2e^{-\bar{\gamma}\frac{\epsilon^2}{4}\frac{\l}{n\log n}}$.

{Third, we bound the Lipschitz constant of $G_{-\epsilon}$. For $\htilde, \xtilde \in \mathbb{R}^{n}$ and $\mtilde, \ytilde \in \mathcal{S}^{p-1}$ we have}
\begin{align*}
	& G_{-\epsilon}(\vh,\vx,\vm,\vy)-G_{-\epsilon}(\htilde,\xtilde,\mtilde,\ytilde)\\
	=& \sum_{i=1}^{\l}\Big[t_{-\epsilon}(\bit\vh) t_{-\epsilon}(\bit\vx)(\cit\vm)_+ (\cit\vy)_+-t_{-\epsilon}(\bit\htilde) t_{-\epsilon}(\bit\xtilde)(\cit\mtilde)_+ (\cit\ytilde)_+\Big]\bi\bit\\
	=&\sum_{i=1}^\l\Big[t_{-\epsilon}(\bit\vh) t_{-\epsilon}(\bit\vx)(\cit\vm)_+\left((\cit\vy)_+-(\cit\ytilde)_+\right)\\
	&\quad+t_{-\epsilon}(\bit\vh) t_{-\epsilon}(\bit\vx)(\cit\ytilde)_+\left((\cit\vm)_+-(\cit\mtilde)_+\right)\\
	&\quad+t_{-\epsilon}(\bit\vh)(\cit\mtilde)_+(\cit\ytilde)_+\left( t_{-\epsilon}(\bit\vx)-t_{-\epsilon}(\bit\xtilde)\right)\\
	&\quad+t_{-\epsilon}(\bit\xtilde)(\cit\mtilde)_+(\cit\ytilde)_+\left( t_{-\epsilon}(\bit\vh)-t_{-\epsilon}(\bit\htilde)\right)\Big]\bi\bit\\
	=&\mB^\intercal\Big[\diag\left(t_{-\epsilon}(\mB\vh)\odot t_{-\epsilon}(\mB\vx) \odot(\mC\vm)_+\right)\diag( (\mC\vy)_+-(\mC\ytilde)_+)\\
	&\quad +\diag\left(t_{-\epsilon}(\mB\vh)\odot t_{-\epsilon}(\mB\vx)\odot (\mC\ytilde)_+\right)\diag\left( (\mC\vm)_+-(\mC\mtilde)_+\right)\\
	&\quad +\diag\big(t_{-\epsilon}(\mB\vh)\odot (\mC\mtilde)_+\odot (\mC\ytilde)_+\big)\diag\left( t_{-\epsilon}(\mB\vx)-t_{-\epsilon}(\mB\xtilde)\right)\\
	&\quad +\diag\left(t_{-\epsilon}(\mB\xtilde)\odot (\mC\mtilde)_+\odot ( \mC\ytilde)_+\right)\diag\left( t_{-\epsilon}(\mB\vh)-t_{-\epsilon}(\mB\htilde)\right)\Big]\mB
\end{align*}
Thus,
\begin{align*}
	&\|G_{-\epsilon}(\vh,\vx,\vm,\vy)-G_{-\epsilon}(\htilde,\xtilde,\mtilde,\ytilde)\|\\
	\leq &\|\mB\|^2\Big[\|\mC\vm\|_\infty\|(\mC\vy)_+-(\mC\ytilde)_+\|_\infty+\|\mC\ytilde\|_\infty\| (\mC\vm)_+-(\mC\mtilde)_+\|_\infty\\
	&\quad+\|\mC\mtilde\|_\infty\|\mC\ytilde\|_\infty\|t_{-\epsilon}(\mB\vx)-t_{-\epsilon}(\mB\xtilde)\|_\infty+\|\mC\mtilde\|_\infty\|\mC\ytilde\|_\infty\|t_{-\epsilon}(\mB\vh)-t_{-\epsilon}(\mB\htilde)\|_\infty\Big]\\
	\leq &\|\mB\|^2\Bigg[\max_{i\in[\l]}\|\ci\|_2\max_{i\in[\l]}\left|(\cit\vy)_+-(\cit\ytilde)_+\right|+\max_{i\in[\l]}\|\ci\|_2\max_{i\in[\l]}\left|(\cit\vm)_+-(\cit\mtilde)_+\right|\\
	&\ +\left(\max_{i\in[\l]}\|\ci\|_2\right)^2\max_{i\in[\l]}\left|t_{-\epsilon}(\bit\vx)-t_{-\epsilon}(\bit\xtilde)\right|+\left(\max_{i\in[\l]}\|\ci\|_2\right)^2\max_{i\in[\l]}\left|t_{-\epsilon}(\bit\vh)-t_{-\epsilon}(\bit\htilde)\right|\Bigg]\\
	\leq &\|\mB\|^2\Bigg[\max_{i\in[\l]}\|\ci\|_2\max_{i\in[\l]}\left|\cit(\vy-\ytilde)\right|+\max_{i\in[\l]}\|\ci\|_2\max_{i\in[\l]}\left|\cit(\vm-\mtilde)\right|\\
	&\quad+\left(\max_{i\in[\l]}\|\ci\|_2\right)^2\max_{i\in[\l]}\frac{1}{\epsilon}\left|\bit(\vx-\xtilde)\right|+\left(\max_{i\in[\l]}\|\ci\|_2\right)^2\max_{i\in[\l]}\frac{1}{\epsilon}\left|\bit(\vh-\htilde)\right|\Bigg]\\
	\leq &\|\mB\|^2\Bigg[\left(\max_{i\in[\l]}\|\ci\|_2\right)^2\left\|\vy-\ytilde\right\|+\left(\max_{i\in[\l]}\|\ci\|_2\right)^2\left\|\vm-\mtilde\right\|\\
	&\quad+\frac{1}{\epsilon}\left(\max_{i\in[\l]}\|\ci\|_2\right)^2\max_{i\in[\l]}\|\bi\|_2\left\|\vx-\xtilde\right\|+\frac{1}{\epsilon}\left(\max_{i\in[\l]}\|\ci\|_2\right)^2\max_{i\in[\l]}\|\bi\|_2\left\|\vh-\htilde\right\|\Bigg]\\
	\leq & \|\mB\|^2\Big[9p\left\|\vy-\ytilde\right\|+9p\left\|\vm-\mtilde\right\|+\frac{27}{\epsilon}\sqrt{n}p\left\|\vx-\xtilde\right\|+\frac{27 }{\epsilon}\sqrt{n}p\left\|\vh-\htilde\right\|\Big]
\end{align*}
where the first inequality follows because $|t_{-\epsilon}(z)|\leq 1$ for all $z$, and the third inequality follows because $t_{-\epsilon}(z)$ is $1/\epsilon$-Lipschitz and $(z)_+$ is $1$-Lipschitz. { Let $E_1$ be the event that $\|\mB\|\leq 3\sqrt{\l}$. By Corollary 5.35 in \cite{vershynin10in}, for $\mA \in \R^{\l\times n}$ with rows of $\mA$ following $\mathcal{N}(\vzero,\mI_{n})$, we have $\mathbb{P}(\|\mA\|\leq 3\sqrt{\l})\geq 1-2e^{-\l/2}$, if $\l\geq n$. As rows of $\mB$ are truncated, we have $\mathbb{P}(E_1)\geq 1-2e^{-\l/2}$, if $\l\geq n$ as well.} On $E_1$, we have
\begin{align}
	&\|G_{-\epsilon}(\vh,\vx,\vm,\vy)-G_{-\epsilon}(\htilde,\xtilde,\mtilde,\ytilde)\|\notag\\
	\leq&\frac{729\l\sqrt{n}p}{\epsilon}\left[\left\|\vy-\ytilde\right\|+\left\|\vm-\mtilde\right\|+\left\|\vx-\xtilde\right\|+\left\|\vh-\htilde\right\|\right]\label{eq:joint_Lipschitz_pos}
\end{align}
for all $\htilde, \xtilde \in \mathcal{S}^{n-1}$ and $\mtilde, \ytilde \in \mathcal{S}^{p-1}$.

Finally, we complete the proof by a covering argument. Let $\mathcal{N}_\delta$ be a $\delta$-net on $\mathcal{S}^{n-1}\times\mathcal{S}^{p-1}$ such that $|\mathcal{N}_\delta|\leq (3/\delta)^{n+p}$. Take $\delta = \frac{\epsilon^2}{2916\sqrt{n}p}$. Combining \eqref{eq:joint_expectation_pos} and \eqref{eq:joint_concentration_pos}, we have
\begin{align*}
	\forall (\vh,\vm), (\vx,\vy) \in \mathcal{N}_\delta,\quad G_{-\epsilon}(\vh,\vx,\vm,\vy)\preceq&\mathbb{E}G_{-\epsilon}(\vh,\vx,\vm,\vy) + \l\epsilon I_{n}\\
	\preceq&\alpha^2 \l \mQ_{\vh,\vx}\mtrans \mQ_{\vm,\vy}\vy+3\l\epsilon I_{n}{\change.}
\end{align*}
with probability at least
\begin{equation*}
1-2|\mathcal{N}_\delta|e^{-\gamma_K\epsilon^2\frac{\l}{4n\log n}}\geq 1-2\left(\frac{3}{\delta}\right)^{n+p}e^{-\gamma_K\epsilon^2\frac{\l}{4n\log n}}\geq 1-2e^{-\gamma_K\epsilon^2\frac{\l}{4n\log n}+(n+p)\log(3\cdot 2916\sqrt{n}p/\epsilon^2)}.
\end{equation*}
If $\l\geq \tilde{c}n(n+p)\log(np)$ for some $\tilde{c} =\Omega(\epsilon^2\log \epsilon)$, then this probability is at least $1 - 2e^{-\tilde{\gamma}\l}$ for some $\tilde{\gamma} = O(\epsilon^2)$. For $(\vh,\vm), (\vx,\vy) \in \mathcal{S}^{n-1}\times\mathcal{S}^{p-1}$, let $(\htilde,\mtilde), (\xtilde,\ytilde) \in \mathcal{N}_\delta$ be such that $\|\vh-\htilde\|_2\leq\delta$, $\|\vx-\xtilde\|_2\leq\delta$, $\|\vm-\mtilde\|_2\leq\delta$ and $\|\vy-\ytilde\|_2\leq\delta$. By \eqref{eq:joint_Lipschitz_pos}, we have that
\begin{align*}	
\forall (\vh,\vx)\neq(\boldsymbol{0},\boldsymbol{0}) \text{ and } \vm,\vy \in \mathcal{S}^{p-1},\quad  &G_{-\epsilon}(\vh,\vx,\vm,\vy)\\
\preceq & G_{-\epsilon}(\htilde,\xtilde,\mtilde,\ytilde) +\frac{729\l\sqrt{n}p}{\epsilon}4\delta \mI_{n}\\
\preceq & \alpha_1\alpha_2\l \mQ_{\vh,\vx}\mtrans \mQ_{\vm,\vy}\vy+4\l\epsilon \mI_{n}.
\end{align*}
In conclusion, the result of this lemma holds if $\l>(2\bar{c}/\epsilon)^2n^2(\log n)^2$ and $\l\geq \tilde{c}n(n+p)\log(np)$, with probability at least $1-2e^{-\l/2}-2e^{-\tilde{\gamma}\l}>1-4e^{-\gamma \l}$ for some $\gamma = O(\epsilon^2)$ and $\tilde{c} = \Omega(\epsilon^2\log\epsilon)$.
\end{proof}

Next, we now provide an upper bound on the singular values of $G_{\epsilon}(\vh,\vx,\vm,\vy)$.
\begin{lemma}\label{lem:joint_WDC_trunc_neg}
Fix $0<\epsilon<1$. Let ${\mB} \in \R^{\l\times n}$ where $i$th row of $\mB$ satisfy $\bit = \vb^{\intercal} \cdot \vone_{\|\vb\|_2\leq 3\sqrt{n}}$ and $\vb \sim \mathcal{N}(\vzero,\mI_{n})$. Similarly, let ${\mC} \in \R^{\l\times p}$ where $i$th row of $\mC$ satisfy $\cit = \vc^{\intercal} \cdot\boldsymbol{1}_{\|\vc\|_2\leq 3\sqrt{p}}$ and $\vc \sim \mathcal{N}(\vzero,\mI_{p})$. If $\l>c((n\log n)^2+(p\log p)^2)$, then with probability at least $1-4e^{-\gamma \l}$,
\begin{align*}	
&\forall (\vh,\vx)\neq(\boldsymbol{0},\boldsymbol{0}) \text{ and } \vm,\vy \in \mathcal{S}^{p-1},\\
&\quad\qquad G_{\epsilon}(\vh,\vx,\vm,\vy)\succeq \alpha_1\alpha_2\l \mQ_{\vh,\vx}\mtrans \mQ_{\vm,\vy}\vy-4\l\epsilon I_{n}.
\end{align*}
Here, $c$ and $\gamma^{-1}$ are constants that depend only polynomially on $\epsilon^{-1}$ and $\alpha_1$ and $\alpha_2$ as in \eqref{eq:gamma_trunc}.
\end{lemma}
\begin{proof}
	First we bound $\mathbb{E}\left[G_{\epsilon}(\vh,\vx,\vm,\vy)\right]$ for fixed $\vh, \vx \in \mathcal{S}^{n -1}$ and $\vm, \vy \in \mathcal{S}^{p-1}$. Noting that $t_{\epsilon}(z) \geq = \vone_{z>0}(z)-\vone_{0\leq z\leq \epsilon}(z)$, we have
	\begin{align*}
		&\mathbb{E}\left[G_{\epsilon}(\vh,\vx,\vm,\vy)\right]\\
		\succeq &\l\mathbb{E}\Big[\Big(\ind_{\bit\vh\geq0}\ind_{\bit \vx\geq0}(\cit \vm)_+ (\cit \vy)_+ -\left(\ind_{0\leq\bit\vh\leq\epsilon}+\ind_{0\leq\bit\vx\leq\epsilon}\right)(\cit \vm)_+ (\cit \vy)_+\Big)\bi\bit\Big]\\
		=&\l\mathbb{E}\Big[\left(\ind_{\bit\vh\geq0}\ind_{\bit \vx\geq0}\bi\bit\right)(\cit \vm)_+ (\cit \vy)_+\Big]-2\l\mathbb{E}\left[\left(\ind_{0\leq\bit\vh\leq\epsilon}\bi\bit\right)(\cit \vm)_+ (\cit \vy)_+\right].
	\end{align*}
We first note that $\mathbb{E}\left[\ind_{\bit\vh\geq0}\ind_{\bit \vx\geq0}\bi\bit\right] = \alpha_1 \mQ_{\vh,\vx}$ and $\mathbb{E}\left[(\cit\vm)_+(\cit\vy)_+\right] = \alpha_2 \mtrans \mQ_{\vm,\vy}\vy$ where $\alpha_i$ satisfies $0.97<\alpha_i<1$. Also, we have $\left|\mtrans \mQ_{\vm,\vy}\vy\right| \leq \frac{1}{2}$ and $\mathbb{E}\left[\ind_{0\leq\bit\vh\leq\epsilon}\bi\bit\right]\preceq\frac{\epsilon\alpha_1}{2}\mI_{n}$. Thus,
	\begin{align}
	\mathbb{E}\left[G_{\epsilon}(\vh,\vx,\vm,\vy)\right] &\succeq \alpha_1\alpha_2\l\cdot\mtrans \mQ_{\vm,\vy}\vy\cdot \mQ_{\vh,\vx}-2\l\cdot\mathbb{E}\left[\ind_{0\leq\bit\vh\leq\epsilon}\bi\bit\right]\cdot\alpha_2\mtrans \mQ_{\vm,\vy}\vy\notag\\
		&\succeq\alpha_1\alpha_2 \l\cdot\mtrans \mQ_{\vm,\vy}\vy\cdot \mQ_{\vh,\vx}-\frac{\epsilon\alpha_1\alpha_2 \l}{2}\mI_{n}\notag\\
		&\succeq\alpha_1\alpha_2 \l\cdot\mtrans \mQ_{\vm,\vy}\vy\cdot \mQ_{\vh,\vx}-\epsilon \l\mI_{n}\label{eq:joint_expectation_neg}
	\end{align}
	
Second, we show concentration of $G_{\epsilon}(\vh,\vx,\vm,\vy)$ for fixed $\vh, \vx \in \mathcal{S}^{n -1}$ and $\vm, \vy \in \mathcal{S}^{p-1}$ and is similar to the steps shown in proof of Lemma \ref{lem:joint_WDC_trunc_pos}. Let $\vxi_i =\sqrt{t_{\epsilon}(\bit \vh) t_{\epsilon}(\bit \vx)(\cit \vm)_+ (\cit \vy)_+}\bi$. If $\l>(2\bar{c}/\epsilon)^2 n^2(\log n)^2$, we have 
\begin{equation}\label{eq:joint_concentration_neg}
\|G_{\epsilon}(\vh,\vx,\vm,\vy) - \mathbb{E}\left[G_{\epsilon}(\vh,\vx,\vm,\vy)\right]\|\leq \epsilon n
\end{equation}
with probability at least $1- 2e^{-\bar{\gamma}\frac{\epsilon^2}{4}\frac{\l}{n\log n}}$. Here, $\bar{c}$ and $\bar{\gamma}$ are absolute constants. 

Third, we bound the Lipschitz constant of $G_{\epsilon}$, and is again similar to the steps shown in proof of Lemma \ref{lem:joint_WDC_trunc_pos}. If $\l\geq n$ then we have
\begin{align}
	&\|G_{\epsilon}(\vh,\vx,\vm,\vy)-G_{\epsilon}(\htilde,\xtilde,\mtilde,\ytilde)\|\notag\\
	\leq&\frac{729\l\sqrt{n}p}{\epsilon}\left[\left\|\vy-\ytilde\right\|+\left\|\vm-\mtilde\right\|+\left\|\vx-\xtilde\right\|+\left\|\vh-\htilde\right\|\right]\label{eq:joint_Lipschitz_neg}
\end{align}
for all $\htilde, \xtilde \in \mathcal{S}^{n-1}$ and $\mtilde, \ytilde \in \mathcal{S}^{p-1}$ with probability at least $1-2e^{-\l/2}$.

Finally, we complete the proof by a covering argument. Let $\mathcal{N}_\delta$ be a $\delta$-net on $\mathcal{S}^{n-1}\times\mathcal{S}^{p-1}$ such that $|\mathcal{N}_\delta|\leq (3/\delta)^{n+p}$. Take $\delta = \frac{\epsilon^2}{2916\sqrt{n}p}$. Combining \eqref{eq:joint_expectation_neg} and \eqref{eq:joint_concentration_neg}, we have
\begin{align*}
	\forall (\vh,\vm), (\vx,\vy) \in \mathcal{N}_\delta,\quad G_{\epsilon}(\vh,\vx,\vm,\vy)\succeq&\mathbb{E}G_{\epsilon}(\vh,\vx,\vm,\vy) - \l\epsilon I_{n}\\
	\succeq&\alpha^2 \l \mQ_{\vh,\vx}\mtrans \mQ_{\vm,\vy}\vy-3\l\epsilon I_{n}
\end{align*}
with probability at least
\begin{equation*}
1-2|\mathcal{N}_\delta|e^{-\gamma_K\epsilon^2\frac{\l}{4n\log n}}\geq 1-2\left(\frac{3}{\delta}\right)^{n+p}e^{-\gamma_K\epsilon^2\frac{\l}{4n\log n}}\geq 1-2e^{-\gamma_K\epsilon^2\frac{\l}{4n\log n}+(n+p)\log(108\sqrt{n}p/\epsilon^2)}.
\end{equation*}
If $\l\geq \tilde{c}n(n+p)\log(np)$ for some $\tilde{c} =\Omega(\epsilon^2\log \epsilon)$, then this probability is at least $1 - 2e^{-\tilde{\gamma}\l}$ for some $\tilde{\gamma} = O(\epsilon^2)$. For $(\vh,\vm), (\vx,\vy) \in \mathcal{S}^{n-1}\times\mathcal{S}^{p-1}$, let $(\htilde,\mtilde), (\xtilde,\ytilde) \in \mathcal{N}_\delta$ be such that $\|\vh-\htilde\|_2\leq\delta$, $\|\vx-\xtilde\|_2\leq\delta$, $\|\vm-\mtilde\|_2\leq\delta$ and $\|\vy-\ytilde\|_2\leq\delta$. By \eqref{eq:joint_Lipschitz_neg}, we have that
\begin{align*}	
\forall (\vh,\vx)\neq(\boldsymbol{0},\boldsymbol{0}) \text{ and } \vm,\vy \in \mathcal{S}^{p-1},\quad  &G_{\epsilon}(\vh,\vx,\vm,\vy)\\
\succeq & \alpha_1\alpha_2\l \mQ_{\vh,\vx}\mtrans \mQ_{\vm,\vy}\vy-4\l\epsilon \mI_{n}.
\end{align*}
In conclusion, the result of this lemma holds if $\l>(2\bar{c}/\epsilon)^2n^2(\log n)^2$ and $\l\geq \tilde{c}n(n+p)\log(np)$, with probability at least $1-2e^{-\l/2}-2e^{-\tilde{\gamma}\l}>1-4e^{-\gamma \l}$ for some $\gamma = O(\epsilon^2)$ and $\tilde{c} = \Omega(\epsilon^2\log\epsilon)$.
\end{proof}

\subsection{Concentration of matrices with sub-gaussian rows}
The proof of Lemmas \ref{lem:WDC_trunc_pos} and \ref{lem:WDC_trunc_neg} require results from concentration of sub-exponential random variables that has a better dependence on the sub-exponential parameters. To this end, we use the following Bernstein inequality and refer the readers to \cite{Jeong2019Bernstein} for a proof of the theorem.

\begin{theorem}\label{thm:yaniv_berns}
	Let $\va = (a_1,\dots,a_n)$ be a fixed non-zero vector and let $y_1,\dots,y_m$ be independent, mean zero sub-exponential random variables satisfying $\mathbb{E}|y_i|\leq 2$ and $\|y_i\|_{\psi_1}\leq K_i^2\ (K_i\geq 2)$. Then for every $u\geq 0$, we have
	\begin{align*}
		\mathbb{P}\left(\left|\sum_{i=1}^m a_iy_i\right|\geq u\right)\leq 2\exp\left[-c\min\left(\tfrac{u^2}{\sum_{i=1}^m a_i^2K_i^2\log K_i},\tfrac{u}{\|\va\|_\infty K^2\log K}\right)\right],
	\end{align*}
	where $K = \max_{i}K_i$ and $c$ is an absolute constant.
\end{theorem}

We now state a theorem that controls the singular values of a random matrix $\mA$. The Theorem is exactly the same as Theorem 5.39 in \cite{vershynin10in} with the notable difference in the dependence of the constants to the sub-gaussian parameters. We use Theorem \ref{thm:yaniv_berns} to get this improved dependence.
\begin{theorem}\label{thm:singular_concen}
	Let $\mA$ be a $N\times n$ matrix whose rows $\va_i$ are independent sub-gaussian isotropic random vectors in $\R^n$. Then for every $t\geq 0$, with probability at least $1-2\exp(-ct^2)$ one has 
 \begin{align}
 	\sqrt{N} - C\sqrt{n} -t \leq s_{\min}(\mA)\leq s_{\max}(\mA)\leq \sqrt{N}+C\sqrt{n}+t.
 \end{align}
 Here $C = C_K = K\sqrt{\log K}\sqrt{\frac{\log 9}{c_1}}$, $c = c_K = \frac{c_1}{K^2\log K}>0$ with $c_1$ is an absolute constant and $K =\max_i\|\va_i\|_{\psi_2}$.
\end{theorem}
The proof structure of Theorem \ref{thm:singular_concen} is exactly the same as the proof of Theorem 5.39 in \cite{vershynin10in}, and so we provide the proof presented in \cite{vershynin10in} below. 
\begin{proof}
	The proof is a basic version of a covering argunemt, and it has three steps. We need to control $\|\mA\vx\|_2$ for all vectors on the unit sphere. To this end, we discretize the sphere using a $\mathcal{N}$ (the approximation step), establish a tight control of $\|\mA\vx\|_2$ for every fixed vector $\vx \in \mathcal{N}$ with high probability (the concentration step), and finish off by taking a union bound over all $\vx$ in the net. 
	
	{\bf Step 1: Approximation}. Using Lemma 5.36 in \cite{vershynin10in} for the matrix $\mB = \mA/\sqrt{N}$ we see that the conclusion of the theorem is equivalent to 
	\begin{align}\label{eq:versh1}
		\|\tfrac{1}{N}\mA^\intercal\mA - \mI\|\leq \max(\delta,\delta^2)=:\epsilon \text{ where } \delta = C\sqrt{\frac{n}{N}}+\frac{t}{\sqrt{N}}.
	\end{align}
	Using Lemma 5.34 in \cite{vershynin10in}, we can evaluate the operator norm in \eqref{eq:versh1} on a $\tfrac{1}{4}$-net $\mathcal{N}$ pf unit sphere $\mathcal{S}^{n-1}$:
	\begin{equation*}
		\|\frac{1}{N}\mA^\intercal\mA - \mI\|\leq2\max_{\vx\in\mathcal{N}}\left|\left<(\frac{1}{N}\mA^\intercal\mA-\mI)\vx,\vx\right>\right| = 2\max_{\vx\in\mathcal{N}}|\frac{1}{N}\|\mA\vx\|_2^2-1|.
	\end{equation*}
	So to complete the proof it suffices to show that, which high probability,
	\begin{equation}
		\max_{\vx\in\mathcal{N}}|\frac{1}{N}\|\mA\|_2^2-1|\leq\frac{\epsilon}{2}.
	\end{equation}
	By Lemma 5.2 in \cite{vershynin10in}, we can choose the net $\mathcal{N}$ so that it has cardinality $|\mathcal{N}|\leq9^n$.
	
	{\bf Step 2: Concentration} Let us fix any vector $\vx \in \mathcal{S}^{n-1}$. We can express $\|\mA\vx\|_2^2$ as a sum of independent random variablies
	\begin{equation}\label{eq:versh2}
		\|\mA\vx\|_2^2=\sum_{i=1}^N\langle\va_i,\vx\rangle =:\sum_{i=1}^N z_{i}^2
	\end{equation}
	where $\va_i$ denote the rows of the matrix $\mA$. By assumption, $z_i = \langle\va_i,\vx\rangle$ are independent sub-gaussian random variables with $\mathbb{E}z_i^2 = 1$ and $\|z_i\|_{\psi_2}\leq K$. Therefore, by Remark 5.18 and Lemma 5.14 in \cite{vershynin10in}, $z_i^2-1$ are independent centered sub-exponential random variables with $\|z_i^2-1\|_{\psi_1}\leq 2\|z_i^2\|_{\psi_1}\leq 4\|z_i\|_{\psi_2}^2 = 4K^2$.
	
	We can therefore use an exponential deviation inequality, Theorem \ref{thm:yaniv_berns}, to control the sum \eqref{eq:versh2}. 
	
	\begin{align*}
		\mathbb{P}\left\{|\frac{1}{N}\|\mA\|_2^2-1|\geq\frac{\epsilon}{2}\right\} &= \mathbb{P}\left\{|\frac{1}{N}\sum_{i=1}^Nz_i^2-1|\geq\frac{\epsilon}{2}\right\}\\&\leq 2\exp\left[-\tilde{c}_1\min\left(\frac{\epsilon^2N^2/4}{\sum_{i=1}^N4K_i^2\log 2K_i},\frac{\epsilon N/2}{4K\log 2K}\right) \right]\\ &\leq 2\exp\left[-\frac{\tilde{c}_1}{4K^2\log 2K}\min\left(\epsilon^2,\epsilon\right)N \right]\\&= 2\exp\left[-\frac{\tilde{c}_1}{4K^2\log 2K}\delta N \right]\\&\leq 2\exp\left[-\frac{c_1}{K^2\log K}\left(C^2n+t^2\right)\right],
	\end{align*}
	where the last inequality follows by the definition of $\delta$ and using the inequality $(a+b)^2 \geq a^2+b^2$ for $a,b\geq 0$.
	
	{\bf Step 3: Union bound.} Taking the union bound over all vectors $\vx$ in the net $\mathcal{N}$ of cardinality $|\mathcal{N}|\leq 9^n$, we obtain
	\begin{equation*}
		\mathbb{P}\left\{\max_{\vx\in\mathcal{N}}|\frac{1}{N}\|\mA\vx\|_2^2-1|\geq\frac{\epsilon}{2}\right\}\leq 9^n\cdot 2\exp[-\frac{c_1}{K^2\log K}\left(C^2n+t^2\right)]\leq 2\exp[-\frac{c_1}{K^2\log K}t^2],
	\end{equation*}
	where the second inequality follows for $C=C_K$ sufficiently large, e.g. $C = K\sqrt{\log K}\sqrt{\frac{\log 9}{c_1}}$.
\end{proof}
We now state a corollary of Theorem \ref{thm:singular_concen} that applies to general, non-isotropic sub-gaussian distribution.

\begin{cor}\label{cor:singular_noniso}
	Let $\mA$ be a $N\times n$ matrix whose rows $\va_i$ are independent sub-gaussian random vectors in $\R^n$ with second moment matrix $\mSigma$ . Then for every $t\geq 0$, with probability at least $1-2\exp(-ct^2)$ one has 
 \begin{align}
 	\|\tfrac{1}{N}\mA^\intercal\mA - \mSigma \|\leq \max(\delta,\delta^2) \text{ where } \delta = C\sqrt{\frac{n}{N}}+\frac{t}{\sqrt{N}}.
 \end{align}
 Here $C = C_K = K\sqrt{\log K}\sqrt{\frac{\log 9}{c_1}}$, $c = c_K = \frac{c_1}{K^2\log K}>0$ with $c_1$ is an absolute constant and $K =\max_i\|\va_i\|_{\psi_2}$.
\end{cor}

\end{document}